\documentclass[12pt,a4paper]{article}
\usepackage{hyperref}
\usepackage{amsthm,amsmath,amssymb}
\usepackage{verbatim}
\usepackage{graphicx}
\usepackage{url}
\usepackage{xcolor}
\usepackage{float}
\usepackage[T1]{fontenc}
\usepackage{cite}
\usepackage{tikz}
\usetikzlibrary{arrows.meta}

\usepackage[headsepline,markcase=ignoreuppercase]{scrlayer-scrpage}
\automark{section}

\theoremstyle{plain}
\newtheorem{theorem}{Theorem}[section]
\newtheorem{lemma}[theorem]{Lemma}
\newtheorem{proposition}[theorem]{Proposition}
\newtheorem{corollary}[theorem]{Corollary}

\theoremstyle{definition}
\newtheorem{definition}[theorem]{Definition}

\theoremstyle{remark}

\numberwithin{equation}{section}

\def\dom{{\rm dom}}

\def\leqT{\leq_\mathrm{T}}

\def\IN{\mathbb{N}} % natürliche Zahlen
\def\IZ{\mathbb{Z}} % ganze Zahlen
\def\IQ{\mathbb{Q}} % rationale Zahlen
\def\IR{\mathbb{R}} % reelle Zahlen

\def\EC{\mathbf{EC}} % berechenbare Zahlen
\def\NC{\mathbf{NC}} % fast berechenbare Zahlen
\def\LC{\mathbf{LC}} % links-berechenbare Zahlen
\def\WC{\mathbf{WC}} % schwach berechenbare Zahlen
\def\CA{\mathbf{CA}} % rekursiv approximierbare Zahlen

\title{Nearly Computable Real Numbers}

\author{Peter Hertling and Philip Janicki\\
Fakult\"at f\"ur Informatik \\
Universit\"at der Bundeswehr M\"unchen \\
85577 Neubiberg, Germany \\[2mm]
Email: peter.hertling@unibw.de, philip.janicki@unibw.de}

\date{January 30, 2023}

\begin{document}

\maketitle

%\thanks{}

\begin{abstract}
In this article we call a sequence $(a_n)_n$ of elements of a metric space
{\em nearly computably Cauchy} if 
for every strictly increasing computable function $r:\IN\to\IN$ the sequence
$(d(a_{r(n+1)},a_{r(n)}))_n$ converges computably to $0$.
We show that there exists a strictly increasing sequence of rational numbers that
is nearly computably Cauchy and unbounded.
Then we call a real number $\alpha$ {\em nearly computable} if there exists a computable
sequence $(a_n)_n$ of rational numbers that converges to $\alpha$ and is nearly computably Cauchy.
It is clear that every computable real number is nearly computable,
and it follows from a result by Downey and LaForte (2002) that
there exists a nearly computable and left-computable number that is not computable.
We observe that the set of nearly computable real numbers is a real closed field
and closed under computable real functions with open domain,
but not closed under arbitrary computable real functions.
Among other things we strengthen results by Hoyrup (2017) 
and by Stephan and Wu (2005) by showing
that any nearly computable real number that is not computable is weakly $1$-generic
(and, therefore, hyperimmune and not Martin-L{\"o}f random)
and strongly Kurtz random (and, therefore, not $K$-trivial),
and we strengthen a result by Downey and LaForte (2002) by showing that no promptly simple set
can be Turing reducible to a nearly computable real number.

\bigskip

\bigskip

\end{abstract}

\maketitle

\section{Introduction}
\label{section:intro}

A real number $x$ is called {\em computable} if there exists a computable sequence $(q_n)_n$
of rational numbers with $|x - q_n|\leq 2^{-n}$, for all $n$.
From a computability-theoretic point of view this set is certainly the most important subset
of the field of real numbers.
But there are also larger subsets that are of interest from a computability-theoretic point of view.
For example, a real number $x$ is called {\em computably approximable}
if there exists a computable sequence of rational numbers that converges to $x$.
And a real number $x$ is called {\em left-computable}
if there exists a strictly increasing computable sequence of rational numbers that converges to $x$.
Let $\EC$ denote the set of all computable numbers, $\WC$ denote
the set of all differences of two left-computable numbers,
and $\CA$ denote the set of all computably approximable numbers.
It is well known that $\EC \subsetneq \WC \subsetneq \CA$ (see ~\cite{AWZ00}), and that all three subsets
$\EC$, $\WC$, and $\CA$ are real closed subfields of the field $\IR$ of all real numbers
(see~\cite[Theorem 6.3.10]{Wei00} and~\cite{Rai05b}).
The set $\LC$ of left-computable numbers plays an important role in the theory of algorithmic randomness;
see ~\cite{Nie09,DH10}.
And in a series of papers Rettinger and Zheng and other coauthors have analyzed various 
computability-theoretic sets of real numbers, mostly between $\WC$ and $\CA$;
for an overview see \cite{RZ21}.

In this article we consider a new computability-theoretic set of real numbers
that is only slightly larger than the set $\EC$ of computable real numbers.
We call a sequence $(a_n)_n$ of elements of a metric space
{\em nearly computably Cauchy} if it has the property
that for every strictly increasing computable function $r:\IN\to\IN$ the sequence
$(d(a_{r(n+1)},a_{r(n)}))_n$ converges computably to $0$.
We call a real number $\alpha$ {\em nearly computable} if there exists a computable
sequence $(a_n)_n$ of rational numbers that converges to $\alpha$ and is nearly computably Cauchy.
The idea behind this is the following.
The limit of any convergent sequence $(a_n)_n$ of real numbers can be written as the limit
of the series $\sum_{i=0}^\infty b_i$, where $b_i:=a_{i} - a_{i-1}$, for $i>0$, and $b_0:=a_0$.
Then the sequence $(b_n)_n$ converges to zero.
We are interested in the computability-theoretic properties of the real numbers that one obtains
as the limit of such a series when one imposes conditions on the convergence of the sequence $(b_n)_n$.
In this article we consider a strong condition by demanding that for any strictly increasing function
$r:\IN\to\IN$ the sequence $(\sum_{i=r(n)+1}^{r(n+1)} b_i)_n $ converges to zero computably.

In the following two short sections we first introduce some notation and then some basic notions
concerning sequences in metric spaces.
Then we define computable sequences and list some properties of computable sequences.
In Section~\ref{section:nearly-computably-Cauchy} we give various characterizations
of nearly computably Cauchy sequences in a metric space and in a normed vector space.
At the conference CCA 2018, after the presentation of \cite{Jan18},
M.~Schr{\"o}der asked the following question:
is every nearly computably Cauchy sequence necessarily a Cauchy sequence?
On the one hand, we show that this is true for nearly computably Cauchy sequences
that are additionally computable.
On the other hand, we show that there exists a strictly increasing sequence of rational numbers that is
nearly computably Cauchy and unbounded (and, therefore, not a Cauchy sequence).
We also make some observations about nearly computably Cauchy sequences that will turn out to be useful later.
In Section~\ref{section:nearly-computable-field} we study the set of all nearly computable real numbers.
It is a proper superset of the set of all computable real numbers.
We show that it is a real closed field. In fact, it is closed under computable functions with open domain.
But we also show that it is not closed under arbitrary computable functions.
In Section~\ref{section:nearly-computable-left-computable} we give several characterizations
of real numbers that are left-computable and nearly computable.
We note that it follows from a result by Downey and LaForte~\cite{DL02} that
there exists a nearly computable and left-computable number that is not computable.
We also discuss two cases where the requirement that a real number should be nearly computable
and left-computable and satisfy some additional condition forces the real number to be even computable.
For example, an argument by Downey, Hirschfeldt, and LaForte~\cite[Pages 105, 106]{DHL04}
shows that any strongly left-computable and nearly computable real number is even computable.
In Section~\ref{section:nearly-computable-randomness} we discuss three further results of the same type,
one by Hoyrup~\cite{Hoy17} and two by Stephan and Wu~\cite{SW05}.
Hoyrup~\cite{Hoy17} has shown that  any nearly computable and left-computable number
that is not computable is weakly $1$-generic.
We show that in this result the assumption that the real number should be left-computable can be omitted.
As any weakly $1$-generic real is hyperimmune~\cite{Kur83} and not Martin-L{\"o}f random
we obtain a similar strengthening of a result by Stephan and Wu~\cite{SW05},
who had shown that any nearly computable number that is left computable
is hyperimmune (and, therefore, not Martin-L{\"o}f random).
Finally, also in their result that any nearly computable number that is left computable
is strongly Kurtz random (and, therefore, not $K$-trivial) one can omit the assumption
that the real number should be left-computable.
In the final section, Section~\ref{section:promptly-simple},
we strengthen an observation by Downey and LaForte~\cite{DL02} in a similar way
by showing that no promptly simple set can be Turing reducible to a nearly computable real number.
Some of the results of this paper have been presented at the conference CCA 2018~\cite{Jan18},
others at the conference CCA 2022~\cite{HJ22}.

\section{Preliminaries}
\label{section:prelim}

Let $X$ and $Y$ be sets. A subset $f\subseteq X\times Y$ with the property that for any $x\in X$ there is at most
one $y\in Y$ with $(x,y)\in f$ is called a {\em function $f:\subseteq X\to Y$} with
{\em domain $\dom(f):=\{x\in X~:~(\exists y \in Y) (x,y) \in f\}$}. Instead of $(x,y)\in f$ we write $f(x)=y$.
For $x\in X$ we write $f(x)\downarrow$ if $x\in\dom(f)$ and $f(x)\uparrow$ if $x\not\in\dom(f)$.
If $\dom(f)=X$ then we may write the function $f$ as $f:X\to Y$ and call it a {\em total} function.
Let $Y^X$ be the set of all total functions from $X$ to $Y$.
Note that in the case $X=\emptyset$ there is exactly one function in $Y^X$, the empty function.
For $k\in\IN$ let $X^k :=X^{\{0,\ldots,k-1\}}$ be the set of all sequences of length $k$ of elements of $X$.
For $k=0$ this set contains exactly one element, the empty sequence.
Let $X^*:=\bigcup_{k=0}^\infty X^k$ be the set of all finite sequences of elements of $X$.
Let $\IN=\{0,1,2,\ldots\}$ as usual be the set of nonnegative integers or {\em natural numbers}.
The elements of $X^\IN$ are called {\em sequences} with elements in $x$, and a sequence $p:\IN\to X$
is often written in the form $(p(n))_n$ or $(p_n)_n$.
For any $k\in\IN$ and $p\in X^\IN$ let $p\upharpoonright k$ be the prefix of length $k$
of the sequence $p:\IN\to X$,
that is, the uniquely determined element $g\in X^k$ with $g\subseteq p$, in other words,
$p\upharpoonright k=p(0)\ldots p(k-1) = p_0\ldots p_{k-1}$.
Sometimes we will identify a subset $A\subseteq\IN$ with its
{\em characteristic function $\chi_A:\IN\to\{0,1\}$} defined by
\[ \chi_A(n) := \begin{cases}
     1 & \text{if } n \in A, \\
     0 & \text{if } n \not\in A,
     \end{cases} \]
for $n\in\IN$.
The function $\langle \cdot,\cdot\rangle:\IN^2\to\IN$ defined by
$\langle i,j\rangle:= \frac{(i+j)\cdot (i+j+1)}{2} + i$, for all $i,j\in\IN$,
is a computable bijection between $\IN^2$ and $\IN$.
Also its projections $\pi_1:\IN\to\IN$ and $\pi_2:\IN\to\IN$ defined by
$\pi_1(\langle i,j\rangle):=i$ and $\pi_2(\langle i,j\rangle):=j$, for all $i,j\in\IN$,
are computable.
We can extend this to a bijection $\langle \cdots \rangle:\IN^k\to\IN$ for any $k\geq 1$ by defining
$\langle i \rangle:=i$, for any $i$, and by
$\langle m_1,\ldots,m_k,m_{k+1}\rangle := \langle \langle m_1,\ldots,m_k\rangle, m_{k+1}\rangle$,
for $k\geq 2$ and $m_1,\ldots,m_k,m_{k+1} \in\IN$.

Let $\varphi_0,\varphi_1,\varphi_2,\ldots$ be a standard enumeration
of all possibly partial computable functions with domain and range in $\IN$.
For every $e\in\IN$ we write $W_e:=\dom(\varphi_e)$.
Note that then $W_0,W_1,W_2,\ldots$ is an effective list of all c.e.~subsets of $\IN$.
Let $h_u:\IN\to\IN^2$ be an injective computable function with
$h_u(\IN)=\{(e,n)\in\IN^2 \mid \varphi_e(n)\downarrow\}$.
For every $s\in\IN$ we write $W_e[s]:=\{n\in\IN ~:~ (e,n) \in \{h_u(0),\ldots,h_u(s-1)\}\}$.
Furthermore, sometimes for a c.e. set $D$ and for $s\in\IN$ we use the expression $D[s]$.
Then it is understood that we have fixed some $e$ with $D=W_e$ and write $D[s]$ for $W_e[s]$.

Let $\IQ$ be the set of all rational numbers and $\IR$ be the set of all real numbers.
In order to transfer computability notions from $\IN$ to $\IQ$ we use
the following surjective function $\nu_\IQ:\IN\to\IQ$ defined by
\[ \nu_\IQ( \langle a,b \rangle) :=\begin{cases} 
    \frac{-a/2}{b+1} & \text{if } a \text{ is even}, \\
    \frac{(a+1)/2}{b+1} & \text{if } a \text{ is odd},
    \end{cases} \]
for $a,b\in\IN$.
Instead of $\nu_\IQ$ one may use any equivalent numbering of $\IQ$;
see the end of Section~\ref{section:comp-sequences}.
For introducing useful computability notions on $\IR$ even any equivalent numbering
of a dense subset of $\IR$ would work as well; 
see again the end of Section~\ref{section:comp-sequences}.
For any $k\geq 1$ we define a surjective function
$\nu_\IQ^{(k)}:\IN\to \IQ^k$ by 
$\nu_\IQ^{(k)}(\langle i_1,\ldots,i_k\rangle):=(\nu_\IQ(i_1),\ldots,\nu_\IQ(i_k))$, for $i_1,\ldots,i_k\in\IN$
(note that $\nu_\IQ^{(1)}=\nu_\IQ$).

\section{Computably Convergent Sequences and Computably Cauchy Sequences}

A {\em metric space} is a pair $(X,d)$ consisting of a nonempty set $X$ and a function $d:X\times X\to\IR$,
the {\em metric} or {\em distance function}, satisfying the following three conditions:
\begin{enumerate}
\item
$(\forall x,y\in X)\ d(x,y)=0 \iff x=y$, 
\item
$(\forall x,y \in X)\ d(x,y)=d(y,x)$
(symmetry),
\item
$(\forall x,y,z \in X)\ d(x,z) \leq d(x,y) + d(y,z)$
(triangle inequality).
\end{enumerate}

We note that, if $(X,d)$ is a metric space then, for all $x,y \in X$, we have
$d(x,y) \geq 0$
(indeed, $d(x,y)= \frac{1}{2}\cdot (d(x,y)+d(y,x)) \geq \frac{1}{2}\cdot d(x,x)=0$).
We remind the reader of some classical notions and observations for sequences in $X$.

\begin{definition}
Let $(X,d)$ be a metric space.
A sequence $(x_n)_n$ in $X$ \emph{converges to an} element $y \in X$ if 
there exists a total function $g:\IN\to\IN$ such that, 
for all $m,n\in\IN$,
\[ \text{if } m\geq g(n) \text{ then } d\left( x_m, y \right) \leq 2^{-n} . \]
A function $g$ with this property is called a \emph{modulus of convergence of the sequence $(x_n)_n$}.
\end{definition}

It is easy to see and well known that,
if a sequence $(x_n)_n$ in a metric space $(X,d)$ converges to an element $y \in X$
and to an element $z\in X$ then $y=z$.
Thus, if a sequence $(x_n)_n$ in $X$ converges to an element $y$ then this element is uniquely determined.
We call it {\em the limit} of $(x_n)_n$ and write it $\lim_{n\to\infty} x_n$.

\begin{definition}
Let $(X,d)$ be a metric space.
A sequence $(x_n)_n$ in $X$ is a \emph{Cauchy sequence}
if there exists a total function $g:\IN\to\IN$ such that, for all $l,m,n\in\IN$,
\[ \text{if } l,m\geq g(n) \text{ then } d\left( x_l, x_m \right) \leq 2^{-n} . \]
A function $g$ with this property is called a \emph{Cauchy modulus of $(x_n)_n$}.
\end{definition}

The elementary proof of the following lemma is left to the reader.

\begin{lemma}
\label{lemma:conv-Cauchy}
Let $(X,d)$ be a metric space,
Let $(x_n)_n$ be a convergent sequence in $X$.
\begin{enumerate}
\item
If a function $g:\IN\to\IN$ is a modulus of convergence of $(x_n)_n$ then
the function $h:\IN\to\IN$ defined by $h(n):=g(n+1)$, for all $n\in\IN$, is a Cauchy modulus of $(x_n)_n$.
\item
The sequence $(x_n)_n$ in $X$ is a Cauchy sequence.
\item
A Cauchy modulus of $(x_n)_n$ is a modulus of convergence of $(x_n)_n$ as well.
\end{enumerate}
\end{lemma}

A metric space $(X,d)$ is called \emph{complete} if every Cauchy sequence is convergent.
Thus, a sequence $(x_n)_n$ in a complete metric space $(X,d)$ 
is convergent if, and only if, it is a Cauchy sequence.

Based on the notions above we can now define
computable versions of convergence and of the Cauchy condition.

\begin{definition}
Let $(X,d)$ be a metric space.
\begin{enumerate}
\item
A sequence $(x_n)_n$ in $X$ \emph{converges computably} or \emph{is computably convergent}
if it converges and there exists a computable modulus $g:\IN\to\IN$ of convergence of $(x_n)_n$.
\item
A sequence $(x_n)_n$ in $X$ is \emph{computably Cauchy}
if there exists a computable Cauchy modulus of $(x_n)_n$.
\end{enumerate}
\end{definition}

\begin{corollary}
\label{cor:conv-comp-Cauchy}
Let $(X,d)$ be a metric space.
\begin{enumerate}
\item
\label{cor:conv-comp-Cauchy-1}
If a sequence $(x_n)_n$ in $X$ converges computably
then it is computably Cauchy.
\item
\label{cor:conv-comp-Cauchy-2}
If a sequence $(x_n)_n$ in $X$ is computably Cauchy and converges then it converges computably.
\item
\label{cor:conv-comp-Cauchy-3}
A sequence in a complete metric space converges computably if, and only if, it is computably Cauchy.
\end{enumerate}
\end{corollary}

\begin{proof}
The assertions follow from Lemma~\ref{lemma:conv-Cauchy}.
\end{proof}

We end this section with a well-known observation on Cauchy sequences in a normed vector space.
A normed vector space $X$ with norm $|\cdot|$ becomes a metric space
with the metric $d:X^2\to\IR$ defined by $d(x,y) := |x-y|$, for $x,y\in X$.
A sequence $(x_n)_n$ in a normed vector space $X$ with norm $|\cdot|$
is called {\em bounded} if there exists a number
$B\in\IN$ such that $|x_n|\leq B$, for all $n\in\IN$. If a number $B$ with this property does not exist then
$(x_n)_n$ is called {\em unbounded}.

\begin{lemma}
\label{lemma:unbounded-notCauchy}
Any Cauchy sequence in a normed vector space is bounded.
\end{lemma}

We omit the very easy and well known proof.
For example, for any $n\geq 1$ the set $\IR^n$ with the usual Euclidean norm
$|(x_1,\ldots,x_n)|:=\sqrt{ \sum_{i=1}^n x_i^2}$ is a normed vector space.
In the case $n=1$ the norm $|\cdot|$ is simply the usual absolute value function.

\section{Computable Sequences}
\label{section:comp-sequences}

First we have a look at computable sequences of rational numbers and computable sequences
of real numbers. Then we shortly discuss computable sequences in computable metric spaces.

\begin{definition}
\label{definition:computable-rat-seq}
A sequence $(x_n)_n$ of rational numbers is called {\em computable} 
if there exists a total computable function $f:\IN\to\IN$ such that, for all $n\in\IN$, $q_n=\nu_\IQ(f(n))$.
\end{definition}

%Remember that this notion of a computable sequence of rational numbers does actually not depend
%on the choice of the standard numbering $\nu_\IQ$ of $\IQ$.

\begin{definition}
\begin{enumerate}
\item
A real number $x$ is called {\em computable} if there exists a computable sequence
$(q_n)_n$ of rational numbers satisfying, for all $n\in\IN$, $|x - q_n| \leq 2^{-n}$.
Let $\EC$ be the set of all computable numbers.
\item
A real number $x$ is called {\em left-computable} if there exists
an increasing computable sequence $(q_n)_n$ of rational numbers converging to $x$.
Let $\LC$ be the set of all left-computable numbers.
\end{enumerate}
\end{definition}

\begin{lemma}
For a real number $x$ the following conditions are equivalent.
\begin{enumerate}
\item
$x$ is computable.
\item
there exists a computable sequence $(q_n)_n$ of rational numbers converging to $x$ computably.
\item
$x$ and $-x$ are left-computable.
\end{enumerate}
\end{lemma}

This is well-known and easy to see. We omit the proof.

\begin{definition}
\label{definition:computable-real-seq}
A sequence $(x_n)_n$ of real numbers is called {\em computable}
if there exists a computable sequence $(q_n)_n$ of rational numbers
such that, for all $k,n\in\IN$, $|x_n - q_{\langle n,k\rangle}| \leq 2^{-k}$.
\end{definition}

A sequence of rational numbers can be computable as a sequence of rational numbers, so,
computable according to Definition~\ref{definition:computable-rat-seq},
or it can be computable as a sequence of real numbers, so,
computable according to Definition~\ref{definition:computable-real-seq}.
The following proposition clarifies the relation between these two types of sequences of rational numbers.

\begin{proposition}
\label{prop:comp-seq-rat}
\begin{enumerate}
\item
If a sequence $(x_n)_n$ of rational numbers is a computable sequence of rational numbers
as in Definition~\ref{definition:computable-rat-seq}
then it is also a computable sequence of real numbers as in Definition~\ref{definition:computable-real-seq}.
\item
There exists a sequence $(x_n)_n$ of rational numbers that is a computable sequence of real numbers
according to Definition~\ref{definition:computable-real-seq} but not a computable sequence of rational numbers
according to Definition~\ref{definition:computable-rat-seq}.
\end{enumerate}
\end{proposition}

\begin{proof}
For the first assertion, let $(x_n)_n$ be a computable sequence of rational numbers.
Then the sequence $(q_n)_n$ defined by
$q_{\langle n,k\rangle} := x_n$,
for $k,n\in\IN$, is a computable sequence of rational numbers as well,
and it shows as in Definition~\ref{definition:computable-real-seq}
that $(x_n)_n$ is a computable sequence of real numbers.

We come to the proof of the second assertion.
Let $H\subseteq \IN$ be a computably enumerable set that is not decidable.
By the projection theorem there exists a decidable set $A\subseteq\IN^2$ such that 
$H=\{n \in\IN ~:~ (\exists i\in\IN) \ (n,i) \in A\}$.
We define a sequence $(x_n)_n$ of rational numbers by
\[ x_n := \begin{cases}
    0 & \text{if } n \not\in H, \\
    2^{-\min\{i\in\IN ~:~ (n,i) \in A\}} & \text{otherwise},
\end{cases} \]
for $n\in\IN$. 
If the sequence $(x_n)_n$ were a computable sequence of rational numbers
as in Definition~\ref{definition:computable-rat-seq},
then, given $n$, one could compute a natural number $\langle a,b \rangle$ with
$\nu_\IQ(\langle a,b \rangle) = x_n$ and then, due to the equivalence
\[ n \in H \iff x_n \neq 0 \iff a\neq 0 , \]
one could decide whether $n\in H$ or not. But $H$ is not decidable.
We define a sequence $q:\IN\to\IQ$ by
\[ q_{\langle n,k \rangle} := \begin{cases}
     0 & \text{if } (n,i)\not\in A, \text{for all } i \leq k , \\
2^{-\min\{i\in\IN ~:~ (n,i) \in A\}} & \text{otherwise},
\end{cases} \]
for $n,k\in\IN$. It is clear that this sequence is a computable sequence of rational numbers
and it shows as in Definition~\ref{definition:computable-real-seq}
that $(x_n)_n$ is a computable sequence of real numbers.
\end{proof}

Later we shall need the following two simple observations about 
nondecreasing computable sequences of real numbers.
The first one is certainly well known.

\begin{lemma}
\label{lemma:nondecreasing-cc}
Let $(a_n)_n$ be a nondecreasing computable sequence of real numbers
converging to a computable real number.
Then the sequence $(a_n)_n$ converges computably.
\end{lemma}

\begin{proof}
Let $\alpha:=\lim_{n\to\infty} a_n$.
Let $(c_n)_n$ be a computable sequence of rational numbers such that,
for all $n,k\in\IN$, $|a_n - c_{\langle n,k \rangle}| \leq 2^{-k}$.
Then the sequence $(d_n)_n$ of rational numbers defined by $d_n:=c_{\langle n,n+2\rangle}$,
for $n\in\IN$, is computable as well and converges to $\alpha$.
Let $(b_n)_n$ be a computable sequence of rational numbers with $|b_n - \alpha| \leq 2^{-(n+2)}$,
for all $n\in\IN$.
As both $(d_n)_n$ and $(b_n)_n$ converge to the same number and are computable,
there exists  a total, strictly increasing, computable function
$g:\IN\to\IN$ such $|d_{g(n)} - b_{g(n)}| \leq 2^{-(n+2)}$, for all $n\in\IN$.
We claim that $g$ is a modulus of convergence of the sequence $(a_n)_n$.
Indeed, for all $m,n\in\IN$ with $m\geq g(n)$ we obtain
\begin{eqnarray*}
 | a_m - \alpha| %&=& \alpha - a_m \\
   %&\leq& \alpha - a_{g(n)} \\
   &\leq& |a_{g(n)} - \alpha| \\
   &\leq& |a_{g(n)} - d_{g(n)}| + |d_{g(n)} -  b_{g(n)}| + |b_{g(n)} - \alpha| \\
   &\leq& 2^{-(g(n)+2)} + 2^{-(n+2)} + 2^{-(g(n)+2)} \\
   &\leq& 2^{-(n+2)} + 2^{-(n+2)} + 2^{-(n+2)} \\
   &<& 2^{-n} . 
\end{eqnarray*}
\end{proof}

\begin{definition}
A function $s:\IN\to\IN$ {\em tends to $\infty$} if for every $M\in\IN$ there exists some $N\in\IN$ such that
for all $n\in\IN$,
\[ \text{ if } n \geq N \text{ then } s(n) \geq M . \]
\end{definition}

For example any unbounded nondecreasing function $s:\IN\to\IN$ tends to $\infty$.

\begin{lemma}
\label{lemma:compconv-h1}
Let $(x_n)_n$ be a nondecreasing sequence of real numbers.
Let $s:\IN\to\IN$ be a computable function that tends to $\infty$.
If the sequence $(x_{s(n+1)} - x_{s(n)})_n$ converges computably to $0$
then the sequence $(x_{n+1} - x_n)_n$ converges computably to $0$ as well.
\end{lemma}

\begin{proof}
Let us assume that the sequence $(x_{s(n+1)} - x_{s(n)})_n$ converges computably to $0$.
Let $g:\IN\to\IN$ be a computable modulus of convergence of the sequence $(x_{s(n+1)} - x_{s(n)})_n$.
As this sequence converges to $0$ this means
that, for all $m,n\in\IN$,
\[ m \geq g(n) \Rightarrow |x_{s(m+1)} - x_{s(m)}| \leq  2^{-n} . \]
The function $s\circ g$ is total and computable.
We claim that (A) the sequence $(x_{n+1} - x_n)_n$ converges to $0$
and that (B) the function $s \circ g$ is a computable modulus of convergence of this sequence.
Indeed, let us consider two numbers $k,n\in\IN$ satisfying $k \geq s(g(n))$.
Then the number
\[ m:= \max\{i \in\IN ~:~ s(i) \leq k\} \]
exists because $s$ tends to $\infty$.
The number $m$ satisfies $m\geq g(n)$ and $s(m) \leq k < s(m+1)$. We obtain
\[ x_{s(m)} \leq x_k \leq x_{k+1} \leq x_{s(m+1)} , \]
hence,
\[ |x_{k+1} - x_k| \leq |x_{s(m+1)} - x_{s(m)}| \leq 2^{-n}. \]
This shows the claims (A) and (B).
\end{proof}

Let us conclude this section with a short discussion of computable sequences
and elements in a computable metric space.

\begin{definition}
Let $(X,d)$ be a metric space,
and let $\alpha,\beta:\IN\to X$ be sequences in $X$.
\begin{enumerate}
\item
We write $\alpha\leq\beta$ and say that
{\em $\alpha$ is $\beta$-computable} if there exists a total computable function
$g:\IN^2\to\IN$ such that, for all $n,k\in\IN$,
\[ d( \alpha(n), \beta(g(n,k)) ) \leq 2^{-k} . \]
\item
We write $\alpha \equiv \beta$ and say that {\em $\alpha$ and $\beta$ are equivalent}
if $\alpha\leq \beta$ and $\beta\leq\alpha$. 
\item
An element $x\in X$ is called {\em $\alpha$-computable} if there exists a total computable function
$f:\IN\to\IN$ such that, for all $k\in\IN$,
$d( x, \alpha(f(k)) ) \leq 2^{-k}$.
\end{enumerate}
\end{definition}

Note that the relation $\leq$ is reflexive and transitive. Furthermore, $\alpha\leq\beta$
implies that the set $\{\alpha(n)~:~n\in\IN\}$ is a subset of the closure of the set
$\{\beta(n)~:~n\in\IN\}$.
The relation $\equiv$ is an equivalence relation.
Obviously, an element $x\in X$ is $\alpha$-computable iff the constant sequence $\beta:\IN\to X$
defined by $\beta(n):=x$, for all $n\in\IN$, is $\alpha$-computable.

\begin{definition}
A {\em computable metric space} is a triple $(X,d,\alpha)$ such that
the pair $(X,d)$ is a metric space and $\alpha:\IN\to X$ is a sequence such that
the set $\{\alpha(n)~:~n\in\IN\}$ is a dense subset of $X$
and the function 
$n\mapsto d(\alpha(\pi_1(n)),\alpha(\pi_2(n)))$ is a computable sequence of real numbers.
\end{definition}

\begin{definition}
Let $(X,d,\alpha)$ be a computable metric space.
An element $x\in X$ (a sequence $\gamma:\IN\to X$) is called {\em computable}
if it is $\alpha$-computable.
\end{definition}

The transitivity of $\leq$ on sequences in $X$ implies that in this definition $\alpha$ can be
replaced by any equivalent sequence $\beta$ in $X$ without changing the
set of computable elements of $X$ or the set of computable sequences in $X$.

\begin{lemma}
\label{lemma:comp-metric-space-comp-elem}
Let $(X,d,\alpha)$ be a computable metric space.
An element $x\in X$ is computable iff there exists a computable function $f:\IN\to\IN$ such that the
sequence $(\alpha(f(n)))_n$ converges computably to $x$.
\end{lemma}

We omit the simple proof.

\section{Nearly Computably Cauchy Sequences in a Metric Space}
\label{section:nearly-computably-Cauchy}

The following definition is central for this paper.

\begin{definition}
\label{definition:ncc}
Let $(X,d)$ be a metric space.
We say that a sequence $(x_n)_n$ is \emph{nearly computably Cauchy}
if for every total, computable, and strictly increasing function $r:\IN\to\IN$
the sequence of real numbers $(d(x_{r(n+1)},x_{r(n)}))_n$ converges computably to $0$.
\end{definition}

The motivation behind this definition is that for any convergent sequence,
in fact, even for any Cauchy sequence, $(x_n)_n$ the sequence 
$(d(x_{n+1},x_n))_n$ of distances of consecutive elements converges to $0$.
Of course, this applies to any subsequence as well.
In the definition above we consider an effective version of the condition that the
sequence of distances of consecutive elements of any effective subsequence 
converges to $0$.

\begin{lemma}
\label{lemma:cChauchy-ncc-1}
Let $(X,d)$ be a metric space.
If a sequence $(x_n)_n$ in $X$ is computably Cauchy then 
it is nearly computably Cauchy.
\end{lemma}

\begin{proof}
Let $(x_n)_n$ be a sequence in $X$ that is computably Cauchy.
Let $r:\IN\to\IN$ be a total, strictly increasing, computable function.
We show that the sequence $(d(x_{r(n+1)},x_{r(n)}))_n$ converges computably to $0$.
Let $g:\IN\to\IN$ be a computable Cauchy modulus of $(x_n)_n$, that is,
$g$ is a total, computable function satisfying, for all $l,m,n \in\IN$,
\[ l,m\geq g(n) \Rightarrow d(x_l,x_m) \leq 2^{-n} . \]
We claim that $g$ shows that the sequence $(d(x_{r(n+1)},x_{r(n)}))_n$ converges computably to $0$.
Indeed, if $m \geq g(n)$ then $r(m+1)\geq m+1 \geq g(n)$ and $r(m) \geq m \geq g(n)$, hence,
$d(x_{r(m+1)},x_{r(m)}) \leq 2^{-n}$.
\end{proof}

In the following lemma we show that one may modify the definition in several
ways without changing the set of defined objects.
In the lemma we consider for a total function $r:\IN\to\IN$ the
following three conditions:
\begin{enumerate}
\item[(I)]
$r$ is strictly increasing,
\item[(II)]
$r$ is nondecreasing and unbounded,
\item[(III)]
$r$ is nondecreasing.
\end{enumerate}
Note that (I) implies (II) and (II) implies (III). 

\begin{lemma}
\label{lemma:ncc-basic}
Let $(X,d)$ be a metric space,
and let $(x_n)_n$ be a sequence in $X$.
Then the three conditions
\begin{enumerate}
\item[(Y.C)]
for every computable function $r:\IN\to\IN$ satisfying condition $\mathrm{(Y)}$
the sequence $(d(x_{r(n+1)},x_{r(n)}))_n$ converges computably to $0$
\end{enumerate}
for $\mathrm{Y}\in\{\mathrm{I},\mathrm{II},\mathrm{III}\}$
are equivalent.
\end{lemma}

\begin{proof}
The two implications ``$\mathrm{(II.C)}\Rightarrow\mathrm{(I.C)}$''
and ``$\mathrm{(III.C)}\Rightarrow\mathrm{(II.C)}$'' are trivial.
It is therefore sufficient to prove the two implications
``$\mathrm{(I.C)}\Rightarrow\mathrm{(II.C})$'' and ``$\mathrm{(II.C)}\Rightarrow\mathrm{(III.C)}$''.

We prove ``$\mathrm{(I.C)}\Rightarrow\mathrm{(II.C)}$''.
Let us assume that $\mathrm{(I.C)}$ is true.
Let $r:\IN\to\IN$ be a total, nondecreasing, unbounded computable function.
We define a total function $f:\IN\to\IN$ recursively by
$f(0):=0$ and
\[ f(n+1) := \min\{m \in \IN ~:~ r(m)>r(f(n)) \}
\]
for all $n\in\IN$.
It is clear that $f$ is well-defined, strictly increasing and computable.
Note that 
\[ r(f(n)) = r(f(i)) = r(f(n+1)-1) , \]
for all $n\in\IN$ and all $i\in\{f(n),\ldots,f(n+1)-1\}$.
The function $s:=r\circ f:\IN\to\IN$ is a total, computable, and strictly increasing function.
By $\mathrm{(I.C)}$ the sequence $(d(x_{s(n+1)},x_{s(n)}))_n$ converges computably to $0$.
Let $g:\IN\to\IN$ be a computable modulus of convergence
of the sequence $(d(x_{s(n+1)},x_{s(n)}))_n$, that is, $g:\IN\to\IN$
is a total, computable function such that, for all $m,n\in\IN$,
\[ m\geq g(n) \Rightarrow d( x_{s(m+1)}, x_{s(m)}) \leq 2^{-n} . \]
The function $h:=f\circ g$ is total and computable as well.
We claim that it is a computable modulus of convergence of the sequence
$(d(x_{r(n+1)},x_{r(n)}))_n$ and that this sequence converges to $0$, that is,
for all $m,n\in\IN$
\[ m\geq h(n) \Rightarrow d(x_{r(m+1)},x_{r(m)}) \leq 2^{-n} . \]
Let $m\geq h(n)$.
As $r$ is nondecreasing, either $r(m+1)=r(m)$ or $r(m+1)>r(m)$.
In the first case we clearly have $d(x_{r(m+1)},x_{r(m)})=0$.
In the second case there is some $k\in\IN$ with $m+1=f(k)$.
In fact, $k>0$ as $f(0)=0$.
We observe $r(m+1)=r(f(k)) = s(k)$ and
$r(m) = r(f(k)-1) = r(f(k-1)) = s(k-1)$.
As $f$ is strictly increasing, from 
\[ f(k)=m+1 > m \geq h(n) = f(g(n)) \]
we conclude $k>g(n)$, hence, $k-1 \geq g(n)$.
Finally, we obtain
\[ d(x_{r(m+1)},x_{r(m)})  = d(x_{s(k)},x_{s(k-1)}) \leq 2^{-n}. 
\]
We have shown that $\mathrm{(I.C)}$ implies $\mathrm{(II.C)}$.

Now, we prove ``$\mathrm{(II.C)}\Rightarrow\mathrm{(III.C)}$''.
Let us assume that $\mathrm{(II.C)}$ is true.
Let $r:\IN\to\IN$ be a total, nondecreasing, computable function.
If $r$ is unbounded then the assertion,
that $(d(x_{r(n+1)},x_{r(n)}))_n$ converges computably to $0$,
is true by assumption.
Let us assume that $r$ is bounded. Then $r$ is ultimately constant, that is, there
exists a number $N\in\IN$ such that 
$r(n)=r(N)$, for all $n\geq N$.
Then $x_{r(n+1)}=x_{r(n)}$, for all $n\geq N$.
Hence, the constant function $g:\IN\to\IN$ defined by $g(n):=N$, for all $n\in\IN$, shows that
the sequence $(d(x_{r(n+1)},x_{r(n)}))_n$ converges computably to $0$.
\end{proof}

In the following lemma we consider the special case of a normed vector space and show that
in this case the condition in Definition~\ref{definition:ncc} can be modified in further ways
without changing the set of defined objects.
In the following lemma for a sequence $(y_n)_n$ in a normed vector space
we consider the following three conditions:
\begin{enumerate}
\item[(A)]
$(y_n)_n$ is computably Cauchy,
\item[(B)]
$(y_n)_n$ converges computably,
\item[(C)]
$(y_n)_n$ converges computably to the zero in the normed vector space.
\end{enumerate}
Note that (C) implies (B) and by Corollary~\ref{cor:conv-comp-Cauchy}.\ref{cor:conv-comp-Cauchy-1}
(B) implies (A).
Note also that a sequence $(y_n)_n$ converges computably to the zero in the normed vector space
iff the sequence $(|y_n|)_n$ converges computably to the real number $0$
(where for $y$ in the normed vector space by $|y|$ we denote its norm).

\begin{lemma}
\label{lemma:ncc-basic-2}
Let $X$ be a normed vector space with norm $|\cdot |$,
and let $(x_n)_n$ be a sequence in $X$.
Then the nine conditions
\begin{enumerate}
\item[(Y.Z)]
for every computable function $r:\IN\to\IN$ satisfying condition $\mathrm{(Y)}$
the sequence $(x_{r(n+1)}-x_{r(n)})_n$ satisfies condition $\mathrm{(Z)}$,
\end{enumerate}
for $\mathrm{Y}\in\{\mathrm{I},\mathrm{II},\mathrm{III}\}$
and $\mathrm{Z} \in \{\mathrm{A},\mathrm{B},\mathrm{C}\}$ are equivalent.
\end{lemma}

\begin{proof}
The three implications ``$\mathrm{(II.Z)}\Rightarrow\mathrm{(I.Z)}$'' and
the three implications ``$\mathrm{(III.Z)}\Rightarrow\mathrm{(II.Z)}$'',
for $\mathrm{Z} \in \{\mathrm{A},\mathrm{B},\mathrm{C}\}$ are trivial.
The three implications ``$\mathrm{(Y.B)}\Rightarrow\mathrm{(Y.A)}$'' 
for $\mathrm{Y} \in \{\mathrm{I},\mathrm{II},\mathrm{III}\}$ follow from
Corollary~\ref{cor:conv-comp-Cauchy}.\ref{cor:conv-comp-Cauchy-1}, and
the three implications ``$\mathrm{(Y.C)}\Rightarrow\mathrm{(Y.B)}$'' 
for $\mathrm{Y} \in \{\mathrm{I},\mathrm{II},\mathrm{III}\}$ are trivial.
The equivalence of the three conditions $\mathrm{(I.C)}$, $\mathrm{(II.C)}$,
and $\mathrm{(III.C)}$ follows from Lemma~\ref{lemma:ncc-basic} because
a sequence $(x_{r(n+1)}-x_{r(n)})_n$ converges computably to the zero in $X$
iff the sequence $(|x_{r(n+1)} - x_{r(n)}|)_n$ converges computably to the real number $0$.

It is therefore sufficient to prove the implication
``$\mathrm{(I.A)}\Rightarrow\mathrm{(I.C)}$''.
Let us assume that $\mathrm{(I.A)}$ is true.
Let $r:\IN\to\IN$ be a total, strictly increasing, computable function.
We define $y_n:=x_{r(n+1)}-x_{r(n)}$, for all $n\in\IN$.
By $\mathrm{(I.A)}$ the sequence $(y_n)_n$ is computably Cauchy.
We wish to show that it converges computably to $0 \in X$.
According to Corollary~\ref{cor:conv-comp-Cauchy}.\ref{cor:conv-comp-Cauchy-2}
it is sufficient to show that it converges to $0\in X$.
For the sake of a contradiction, let us assume that it does not converge to $0$.
Then there exists a real number $\varepsilon>0$ such that
$|y_n|>\varepsilon$ for infinitely many $n$.
Remember that the sequence $(y_n)_n$ is computably Cauchy, hence, Cauchy.
We choose a number $N\in\IN$ such that $|y_N|>\varepsilon$ and, for all $m,n \geq N$,
$|y_m-y_n|\leq \varepsilon/2$.
We define $z:=y_N$. Then $|z| > \varepsilon$ and, for all $n\geq N$,
$|y_n - z|\leq \varepsilon/2$.
We define a total, strictly increasing, and computable function
$t:\IN\to\IN$ by 
\[ t(n):= r \left( \frac{n(n+1)}{2} \right) , \]
for all $n\in\IN$.
On the one hand, by (I.A) the sequence
$(x_{t(n+1)}-x_{t(n)})_n$ is computably Cauchy, hence, Cauchy.
On the other hand, for all sufficiently large $i$, we have
\[ \left| x_{r(i+1)}-x_{r(i)} - z \right| = |y_i - z| \leq \varepsilon/2. \]
We obtain for all sufficiently large $n$
\begin{eqnarray*}
 \left|x_{t(n+1)}-x_{t(n)} - (n+1)\cdot z \right|
  &=& \left|x_{r \left( \frac{(n+1)(n+2)}{2} \right)} - x_{r \left( \frac{n(n+1)}{2} \right)}
  - (n+1)\cdot z \right| \\
  &=& \left| \sum_{i=0}^n \left( x_{r \left(i+1+ \frac{n(n+1)}{2} \right)}
      - x_{r \left(i+ \frac{n(n+1)}{2} \right)}\right)  - (n+1)\cdot z\right| \\
  &\leq & \sum_{i=0}^n \left| x_{r \left(i+1+ \frac{n(n+1)}{2} \right)}
      - x_{r \left(i+ \frac{n(n+1)}{2} \right)}  - z\right| \\
  &\leq & (n+1) \cdot \varepsilon/2 \\
  &<& \frac{n+1}{2} \cdot |z| .
\end{eqnarray*}
Hence, for all sufficiently large $n$ we obtain
\[ \left|x_{t(n+1)}-x_{t(n)}\right| \geq \frac{n+1}{2} \cdot |z| . \]
Thus, the sequence $(x_{t(n+1)}-x_{t(n)})_n$ is unbounded.
This is in contradiction to the fact that the sequence
$(x_{t(n+1)}-x_{t(n)})_n$ is a Cauchy sequence; compare Lemma~\ref{lemma:unbounded-notCauchy}.
We have shown that the sequence $(x_{r(n+1)}-x_{r(n)})_n$ converges to $0$.
Thus, we have shown that $\mathrm{(I.A)}$ implies $\mathrm{(I.C)}$.
\end{proof}

In Lemma~\ref{lemma:cChauchy-ncc-1} we have seen that any computably Cauchy sequence
is nearly computably Cauchy.
At the conference CCA 2018, after the presentation of \cite{Jan18},
M.~Schr{\"o}der asked the following question:
is every nearly computably Cauchy sequence necessarily a Cauchy sequence?
Perhaps at least if it is an increasing sequence of real or rational numbers?
It turns out that for the answer it is important whether $(x_n)_n$ is computable or not.

\begin{proposition}
Let $(X,d,\alpha)$ be a computable metric space.
Let $(x_n)_n$ be a computable sequence in $X$.
If it is nearly computably Cauchy then it is Cauchy.
\end{proposition}

\begin{proof}
Let $(x_n)_n$ be a computable sequence in $X$ 
that is nearly computably Cauchy.
We wish to show that it is a Cauchy sequence.
For the sake of a contradiction, let us assume that this is not the case.
Then there exists an integer $n_0$ such that for any $N\in\IN$
there are integers $k,l$ with $N \leq k < l$ and $d(x_k,x_l) > 5\cdot 2^{-n_0}$.
There exists a computable function $f:\IN\to\IN$ such that
$d(\alpha_{f(n)},x_n) < 2^{-n_0}$, for all $n\in\IN$.
Let $t:\IN\to\IN$ be defined as follows. Let $t(0), t(1)$ 
be the smallest pair of numbers $k,l$ with $k<l$ and $d(\alpha_{f(k)},\alpha_{f(l)}) > 3\cdot 2^{-n_0}$
(to be precise, that means that $t(0)$ is the smallest number $k$
such that there exists a number $l$ with $k<l$ and $d(\alpha_{f(k)},\alpha_{f(l)}) > 3\cdot 2^{-n_0}$,
and $t(1)$ is the smallest number $l$ with these properties).
And for $n>0$ let $t(2n),t(2n+1)$ be the smallest pair of numbers $k,l$ with
$t(2n-1) < k < l$ and $d(\alpha_{f(k)},\alpha_{f(l)}) > 3\cdot 2^{-n_0}$.
It is clear that $t$ is well-defined, total, strictly increasing, and computable.
Furthermore, for all $n\in\IN$,
$d(\alpha_{f(t(2n))},\alpha_{f(t(2n+1))}) > 3\cdot 2^{-n_0}$, hence,
$d(x_{t(2n)},x_{t(2n+1)}) > 2^{-n_0}$.
But as $(x_n)_n$ is nearly computably Cauchy
the sequence $(d(x_{t(n+1)},x_{t(n)}))_n$ should converge to $0$. Contradiction.
\end{proof}

In particular, if a computable sequence of real numbers is nearly computably Cauchy then it converges.
But in the general case the answer to the question formulated before the previous proposition is no.

\begin{theorem}
\label{theorem:increasing-ncc}
There exists a strictly increasing and unbounded sequence of rational numbers that
is nearly computably Cauchy.
\end{theorem}

Note that by Lemma~\ref{lemma:unbounded-notCauchy}
an unbounded sequence cannot be a Cauchy sequence, 
hence, it cannot be convergent.
For the proof of this theorem we first formulate a lemma.
For this lemma we need some further well-known computability-theoretic notions.
The notion of a computable function $H:\subseteq\IN^\IN\to\IN^\IN$ is used in the usual sense;
for a precise definition see, e.g.,~\cite[Section 9.2.3]{Sch21}.
Let $\rho_{\IQ^\IN}:\IN^\IN\to\IQ^\IN$ be defined by
$\rho_{\IQ^\IN}(p):=(\nu_\IQ(p_0),\nu_\IQ(p_1),\nu_\IQ(p_2),\ldots)$,
for $p=(p_i)_i \in\IN^\IN$.
A function $F:\subseteq \IQ^\IN\times\IN\times\IN^\IN\to\IQ^\IN$ is called {\em computable} if
there exists a computable function $H:\subseteq\IN^\IN\to\IN^\IN$ with
\begin{eqnarray*}
\dom(H) &=& \{p\in\IN^\IN ~:~ (\rho_{\IQ^\IN}(p_1p_3p_5\ldots), p_0, (p_2p_4p_6\ldots)) \in\dom(F) \} \\
\rho_{\IQ^\IN}(H(p)) &=& F(\rho_{\IQ^\IN}(p_1p_3p_5\ldots), p_0, (p_2p_4p_6\ldots)),
\end{eqnarray*}
for $p\in\dom(H)$.
Similarly, the notion of a computable function $G:\subseteq \IQ^\IN\times\IN\times\IN^\IN\to\IN^\IN$
is defined.

\begin{lemma}
\label{lemma:nondecreasing-ncc}
There exist computable functions $F:\subseteq \IQ^\IN\times\IN\times\IN^\IN\to\IQ^\IN$ and
$G:\subseteq \IQ^\IN\times\IN\times\IN^\IN\to\IN^\IN$ such that for any triple
$(a, N, s)$ consisting of
\begin{itemize}
\item
a strictly increasing, unbounded sequence $a:\IN\to\IQ$ of rational numbers,
\item
a natural number $N$, and
\item
a strictly increasing function $s:\IN\to\IN$
\end{itemize}
both $b:=F(a, N, s)$ and $g:= G(a, N, s)$ are defined and satisfy the following conditions:
\renewcommand\theenumi{\Roman{enumi}}
\begin{enumerate}
\item
$b:\IN\to\IQ$ is a strictly increasing, unbounded sequence of rational numbers,
\item
for all $k\in\IN$, if $a_k\leq N$ then $b_k=a_k$,
\item
for all $k,l\in\IN$, if $k \leq l$ then $b_l - b_k \leq a_l - a_k$,
\item
the sequence $(b_{s(k+1)} - b_{s(k)})_k$ converges to $0$
and the function $g:\IN\to\IN$ is a modulus of convergence of this sequence.
\end{enumerate}
\renewcommand\theenumi{\arabic{enumi}}
\end{lemma}

Conditions (II) and (III) mean that the sequence $b$ is obtained from the sequence $a$
by ``compressing'' the part of $a$ with entries larger than $N$ and by leaving the other part unchanged.

\begin{proof}
Let a strictly increasing, unbounded sequence $a:\IN\to\IQ$ of rational numbers,
a natural number $N$, and
a strictly increasing function $s:\IN\to\IN$ be given.
We construct the desired sequence $b:\IN\to\IQ$ and the desired function $g:\IN\to\IN$ step by step.
Note that we will ensure that $b_{s(g(n))} \geq N+n$, for all $n$.
This implies that the new sequence $(b_n)_n$ will be unbounded.
Furthermore, the constructed function $g$ will be strictly increasing.
First, we define
\[ g(0) := \min\{k \in \IN ~:~ a_{s(k)} \geq N \} \]
and $b_j:=a_j$, for all $j\leq s(g(0))$. 
Note that this ensures that $b_{s(g(0))} \geq N$ and that Condition (II) is satisfied.
Let us assume that for some $n\in\IN$ we have already defined
$g(0),\ldots,g(n)$ and $b_0,\ldots, b_{s(g(n))}$ such that
\begin{itemize}
\item
$g(0)< \ldots < g(n)$,
\item
$b_0 < \ldots < b_{s(g(n))}$, 
\item
for all $m\leq n$, $b_{s(g(m))} \geq N+m$,
\item
for all $k,l\leq s(g(n))$, if $k\leq l$ then $b_l - b_k \leq a_l - a_k$,
\item
for all $m< n$ and all $i$ with $g(m)\leq i < g(n)$, 
$b_{s(i+1)} - b_{s(i)} \leq 2^{-m}$.
\end{itemize}
We define
\[ g(n+1) := \min\left\{ k \in \IN ~:~ k > g(n) \text{ and }
     \sum_{i=g(n)}^{k-1}   \min \left\{ 2^{-n}, a_{s(i+1)} - a_{s(i)}\right\} \geq 1 \right\} . \]
This value $g(n+1)$ is well-defined because the sequence $a$ is unbounded.
For $i=g(n),\ldots,g(n+1)-1$ and $j=s(i)+1,\ldots,s(i+1)$ we define
\[  b_j := \begin{cases}
       b_{j-1} + (a_j - a_{j-1}) & \text{if } a_{s(i+1)} - a_{s(i)} \leq 2^{-n} , \\
       b_{j-1} + (a_j - a_{j-1}) \cdot \frac{2^{-n}}{a_{s(i+1)} - a_{s(i)}} & \text{otherwise}.
       \end{cases}.
\]
We have defined $b_{s(g(n))+1}, \ldots, b_{s(g(n+1))}$.
The construction is illustrated in Figure~\ref{figure:one}.
\begin{figure}[htbp]
\begin{center}
\begin{tikzpicture}[scale=1,>=Latex]
\draw[thick] (-2,1) -- (13,1);
\draw[thick] (-2,-1) -- (13,-1);
\draw (-2.4,1) node {$a$};
\draw (-2.4,-1) node {$b$};
\draw (-2,0.8) -- (-2,1.2);
\draw (0,0.8) -- (0,1.2);
\draw (-2,-1.2) -- (-2,-0.8);
\draw (0,-1.2) -- (0,-0.8);
\draw (-2,1.4) node {$0$};
\draw (-2,-1.4) node {$0$};
\draw (0,1.4) node {$N$};
\draw (0,-1.4) node {$N$};
\draw (-1,0.8) node {$\cdots$};
\draw (-1,-0.8) node {$\cdots$};
\draw[->] (-1,0.7) -- (-1,-0.7);
\draw[very thick] (1,0.8) -- (1,1.2);
\draw[very thick] (1,-0.8) -- (1,-1.2);
\draw (1,1.4) node {$a_{s(g(0))}$};
\draw (1,-1.4) node {$b_{s(g(0))}$};
\draw[->] (1,0.7) -- (1,-0.7);
\draw (2.4,0.8) node {$\cdots$};
\draw (1.9,-0.8) node {$\cdots$};
\draw[very thick] (3.8,0.8) -- (3.8,1.2);
\draw[very thick] (2.8,-0.8) -- (2.8,-1.2);
\draw (3.8,1.4) node {$a_{s(g(2))}$};
\draw (2.8,-1.4) node {$b_{s(g(2))}$};
\draw[->] (3.8,0.7) -- (2.8,-0.7);
\draw[very thick] (5.2,0.8) -- (5.2,1.2);
\draw[very thick] (3.8,-0.8) -- (3.8,-1.2);
\draw[->] (5.2,0.7) -- (3.8,-0.7);
\draw[very thick] (4.5,0.9) -- (4.5,1.1);
\draw[very thick] (3.3,-0.9) -- (3.3,-1.1);
\draw[very thick] (5.8,0.8) -- (5.8,1.2);
\draw[very thick] (5.4,0.9) -- (5.4,1.1);
\draw[very thick] (4.0,-0.9) -- (4.0,-1.1);
\draw[very thick] (5.6,0.9) -- (5.6,1.1);
\draw[very thick] (4.2,-0.9) -- (4.2,-1.1);
\draw[very thick] (4.4,-0.8) -- (4.4,-1.2);
\draw[->] (5.8,0.7) -- (4.4,-0.7);
\draw[very thick] (6.0,0.8) -- (6.0,1.2);
\draw[very thick] (4.6,-0.8) -- (4.6,-1.2);
\draw[->] (6.0,0.7) -- (4.6,-0.7);
\draw[very thick] (8.6,0.8) -- (8.6,1.2);
\draw[very thick] (5.6,-0.8) -- (5.6,-1.2);
\draw[->] (8.6,0.7) -- (5.6,-0.7);
\draw[very thick] (6.52,0.9) -- (6.52,1.1);
\draw[very thick] (4.8,-0.9) -- (4.8,-1.1);
\draw[very thick] (7.56,0.9) -- (7.56,1.1);
\draw[very thick] (5.2,-0.9) -- (5.2,-1.1);
\draw[very thick] (9.0,0.8) -- (9.0,1.2);
\draw[very thick] (6.0,-0.8) -- (6.0,-1.2);
\draw[->] (9.0,0.7) -- (6.0,-0.7);
\draw[very thick] (12.0,0.8) -- (12.0,1.2);
\draw[very thick] (7.0,-0.8) -- (7.0,-1.2);
\draw[->] (12.0,0.7) -- (7.0,-0.7);
\draw (12.0,1.4) node {$a_{s(g(3))}$};
\draw (7.0,-1.4) node {$b_{s(g(3))}$};
\draw[very thick] (10.5,0.9) -- (10.5,1.1);
\draw[very thick] (6.5,-0.9) -- (6.5,-1.1);
\draw (2.8,-2) -- (4.6,-2);
\draw (5.0,-2) -- (6.8,-2);
\draw (4.8,-2) node {$1$};
\draw (2.8,-2.2) -- (2.8,-1.8);
\draw (6.8,-2.2) -- (6.8,-1.8);
\end{tikzpicture}
\caption{An illustration of the construction in the proof of Lemma~\ref{lemma:nondecreasing-ncc}.}
\label{figure:one}
\end{center}
\end{figure}
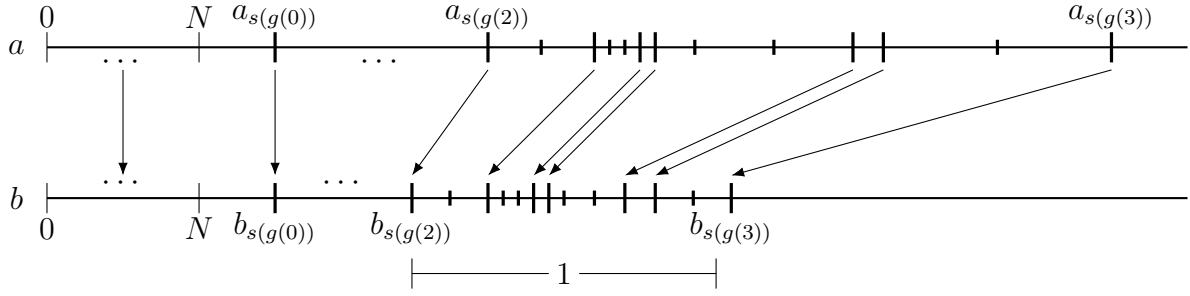
It is clear that $b_0 < \ldots < b_{s(g(n+1))}$ and that,
for all $k,l\leq s(g(n+1))$, if $k\leq l$ then $b_l - b_k \leq a_l - a_k$.
Furthermore, for $i\in\{ g(n),\ldots,g(n+1)-1\}$,
\[ b_{s(i+1)}  - b_{s(i)} = \begin{cases}
         a_{s(i+1)} - a_{s(i)} & \text{if } a_{s(i+1)} - a_{s(i)} \leq 2^{-n} , \\
         2^{-n} & \text{otherwise}.
         \end{cases}
\]
Finally, 
\begin{eqnarray*}
  b_{s(g(n+1))} 
   &=&  b_{s(g(n))} + b_{s(g(n+1))}  - b_{s(g(n))}  \\
   &=& b_{s(g(n))} + \sum_{i=g(n)}^{g(n+1)-1} (b_{s(i+1)}  - b_{s(i)}) \\
   &\geq & N+n + \sum_{i=g(n)}^{g(n+1)-1}   \min \left\{ 2^{-n}, a_{s(i+1)} - a_{s(i)}\right\} \\
   &\geq & N+n+1 . 
%\qedhere
\end{eqnarray*}        
\end{proof}

\begin{proof}[Proof of Theorem~\ref{theorem:increasing-ncc}]
We define a sequence $(q_n)_n$ of rational numbers as follows.
Let $F$ and $G$ be computable functions as in Lemma~\ref{lemma:nondecreasing-ncc}.
Let $(s^{(i)})_i$ be a sequence consisting of all total, strictly increasing, computable functions
$s^{(i)}:\IN\to\IN$ (of course, this sequence is not computable).
We define a sequence $(a^{(i)})_i$ of sequences $a^{(i)}:\IN\to\IQ$ of rational numbers as follows.
Let $a^{(0)}$ be defined by $a^{(0)}_n:=n$, for all $n$.
For all $i\in\IN$ we define 
\[ a^{(i+1)} := F(a^{(i)},i+1,s^{(i)} )  . \]
The sequence $a^{(0)}$ is a strictly increasing, unbounded, computable sequence of rational numbers.
As $F$ is computable, it follows by induction that, for all $i\in\IN$,
the sequence $a^{(i)}$ is a strictly increasing, unbounded, computable sequence of rational numbers.
Finally, we define $q_n:= \lim_{i\to\infty} a^{(i)}_n$, for all $n$.
We claim that the sequence $q:\IN\to\IQ$ is well-defined and has the desired properties.
First, we show that $q$ is well-defined and, for all $i,n\in\IN$, 
\begin{equation}
\label{eq:proposition:nondecreasing-ncc}
 \text{if } n \leq i \text{ then } q_n = a^{(i)}_n .
\end{equation}
As $a^{(0)}_0=0<j$, for all $j>0$, by Condition (II) in Lemma~\ref{lemma:nondecreasing-ncc} we obtain
$a^{(j)}_0=0$ for all $j\in\IN$, by induction.
Let us consider some $n$. Due to Condition (III) in Lemma~\ref{lemma:nondecreasing-ncc}, for all $j\in\IN$,
$a^{(j)}_n = a^{(j)}_n - a^{(j)}_0 \leq a^{(0)}_n - a^{(0)}_0 = n < n+1$, for all $j\in\IN$.
But then, by Condition (II) in Lemma~\ref{lemma:nondecreasing-ncc},
$a^{(i+1)}_n=a^{(i)}_n$, for all $i\geq n$.
This implies that $q=\lim_{i\to\infty} a^{(i)}_n$ is well-defined, for all $n\in\IN$, and that 
\eqref{eq:proposition:nondecreasing-ncc} is true.
The sequence $q$ is strictly increasing because, for all $n\in\IN$,
\[ q_{n+1} - q_n = a^{(n+1)}_{n+1} - a^{(n+1)}_n > 0 . \]
Next, we show that $q$ is unbounded.
For an arbitrary $i$, let $n$ be the largest number with $a^{(i)}_n \leq i+1$. 
Then, by Condition (II) in Lemma~\ref{lemma:nondecreasing-ncc},
$a^{(j)}_n=a^{(i)}_n$, for all $j\geq i$, hence, by \eqref{eq:proposition:nondecreasing-ncc},
$q_n=a^{(\max\{n,i\})}_n = a^{(i)}_n$
and, by Condition (III) in Lemma~\ref{lemma:nondecreasing-ncc},
\[ a^{(i)}_{n+1} - a^{(i)}_n \leq a^{(0)}_{n+1} - a^{(0)}_n = n+1 - n= 1, \]
hence, $q_n = a^{(i)}_n \geq a^{(i)}_{n+1} - 1 > i+1 -1 = i$.
Finally, for all $i\in\IN$ the function $g^{(i)}:= G(a^{(i)},i+1,s^{(i)})$ is computable because $G$ is computable.
We claim that it is a modulus of convergence to $0$ of the sequence $(q_{s^{(i)}(k+1)} - q_{s^{(i)}(k)})_k$.
By Condition (IV) in Lemma~\ref{lemma:nondecreasing-ncc} the function $g^{(i)}$
is a modulus of convergence to $0$ of the sequence
$(a^{(i)}_{s^{(i)}(k+1)} - a^{(i)}_{s^{(i)}(k)})_k$. 
We show that this carries over from $a^{(i)}$ to $q$.
Indeed, let us consider some $m,n\in\IN$ with $m\geq g^{(i)}(n)$. 
We choose some $k\geq \max\{s^{(i)}(m+1),i\}$. 
By~\eqref{eq:proposition:nondecreasing-ncc} and
by Condition (III) in Lemma~\ref{lemma:nondecreasing-ncc} we obtain
\[ q_{s^{(i)}(m+1)} - q_{s^{(i)}(m)} 
   = a^{(k)}_{s^{(i)}(m+1)} -  a^{(k)}_{s^{(i)}(m)}
   \leq a^{(i)}_{s^{(i)}(m+1)} -  a^{(i)}_{s^{(i)}(m)}
   \leq 2^{-n} . \]
We have shown that $q$ is a strictly increasing, unbounded sequence of rational numbers satisfying
Condition (I.C) in Lemma~\ref{lemma:ncc-basic},
thus, it is nearly computably Cauchy.
\end{proof}

In the rest of this section we make several useful observations about
nearly computably Cauchy sequences in a metric space or a normed space
and about nondecreasing sequences of real numbers that are nearly computably Cauchy.
The following lemma will be used several times for sequences $(x_n)_n$ and $(y_n)_n$ 
satisfying even $d(x_n,y_n)\leq 2^{-n}$, for all $n\in\IN$.

\begin{lemma}
\label{lemma:ncc-pair-close}
Let $(X,d)$ be a metric space.
Let $(x_n)_n$ and $(y_n)_n$ be sequences in $X$ such that the sequence of real numbers
$(d(x_n,y_n))_n$ converges computably to $0$.
Then the sequence $(x_n)_n$ is nearly computably Cauchy if, and only if,
the sequence $(y_n)_n$ is nearly computably Cauchy.
\end{lemma}

\begin{proof}
Let us assume that $(x_n)_n$ is nearly computably Cauchy.
By symmetry it is sufficient to show that this implies that $(y_n)_n$ is
nearly computably Cauchy as well.
Let $f:\IN\to\IN$ be a computable modulus of convergence of the sequence $(d(x_n,y_n))_n$.
Let $r:\IN\to\IN$ be a strictly increasing, computable function.
Then the sequence $(d(x_{r(n+1)},x_{r(n)}))_n$ converges computably to $0$.
Let $g:\IN\to\IN$ be a computable modulus of convergence of this sequence.
The function $h:\IN\to\IN$ defined by
$h(n):=\max\{f(n+2),g(n+2)\}$, for $n\in\IN$, is computable. We claim that 
the sequence $(d(y_{r(n+1)},y_{r(n)}))_n$ converges to $0$
and that $h$ is a modulus of convergence of this sequence.
Indeed, for $m,n\in\IN$ with $m\geq h(n)$ we obtain
$r(m+1)\geq m+1 > h(n) \geq f(n+2)$ and $m\geq h(n) \geq g(n+2)$
as well as $r(m)\geq m \geq h(n) \geq f(n+2)$, hence,
\begin{eqnarray*}
d(y_{r(m+1)},y_{r(m)})
 &\leq & d(y_{r(m+1)},x_{r(m+1)}) + d(x_{r(m+1)},x_{r(m)}) + d(x_{r(m)},y_{r(m)}) \\
 &\leq & 2^{-(n+2)} + 2^{-(n+2)} + 2^{-(n+2)} \\
 &<& 2^{-n}.
\end{eqnarray*}
\end{proof}

\begin{lemma}
\label{lemma:nearly-computably-convergent-1}
Let $(X,d)$ be a metric space,
and let $(x_n)_n$ be a sequence in $X$.
Let $s:\IN\to\IN$ be a total, nondecreasing, unbounded computable function,
and define $y_n:=x_{s(n)}$, for all $n$.
If $(x_n)_n$ is nearly computably Cauchy then so is $(y_n)_n$.
\end{lemma}

\begin{proof}
Let $(x_n)_n$ be nearly computably Cauchy.
If $r:\IN\to\IN$ is a total, nondecreasing, unbounded computable function
then $s \circ r:\IN\to\IN$ is a total, nondecreasing, unbounded, computable function as well.
Hence, the sequence
\[ (d(y_{r(n+1)},y_{r(n)}))_n
  = (d(x_{s(r(n+1))},x_{s(r(n))}))_n
  = (d(x_{(s\circ r)(n+1)},x_{(s \circ r)(n)}))_n \]
converges computably.
\end{proof}

\begin{proposition}
\label{prop:ncc-seq-vectorspace}
The set of all nearly computably Cauchy sequences in a normed vector space
is a sub vector space of the vector space of all sequences in the normed space.
\end{proposition}

\begin{proof}
Let $X$ be a normed vector space.
We need to prove the following two assertions:
\begin{enumerate}
\item
If $(x_n)_n$ and $(y_n)_n$ are nearly computably Cauchy sequences
then $(x_n+y_n)_n$ is a nearly computably Cauchy sequence as well.
\item
If $(x_n)_n$ is a nearly computably Cauchy sequence and $\lambda$ is a real number
then $(\lambda \cdot x_n)_n$ is a nearly computably Cauchy sequence.
\end{enumerate}
We prove the first assertion and leave the similar proof of the second assertion to the reader.

Let $(x_n)_n$ and $(y_n)_n$ be nearly computably Cauchy sequences.
Let $r:\IN\to\IN$ be a computable and strictly increasing function.
We claim that the sequence $(x_{r(n+1)} + y_{r(n+1)} - x_{r(n)} - y_{r(n)})_n$ converges computably to $0$.
It is clear that it converges to $0$ because the sequences $(x_{r(n+1)} - x_{r(n)})_n$ and
$(y_{r(n+1)} - y_{r(n)})_n$ converge to $0$ computably.
Let $g_x:\IN\to\IN$ be a computable modulus of convergence of the sequence $(x_{r(n+1)} - x_{r(n)})_n$,
and let $g_y:\IN\to\IN$ be a computable modulus of convergence of the sequence $(y_{r(n+1)} - y_{r(n)})_n$.
Then the function $h:\IN\to\IN$ defined by $h(n):=\max\{g_x(n+1),g_y(n+1)\}$ is computable.
It is a modulus of convergence
of the sequence $(x_{r(n+1)} + y_{r(n+1)} - x_{r(n)} - y_{r(n)})_n$
because, for all $m,n\in\IN$ with $m\geq h(n)$,
\begin{eqnarray*}
|x_{r(m+1)} + y_{r(m+1)} - x_{r(m)} - y_{r(m)}|
 &\leq & |x_{r(m+1)} - x_{r(m)}| + |y_{r(m+1)} - y_{r(m)}| \\
 &\leq & 2^{-(n+1)} + 2^{-(n+1)} \\
 &=& 2^{-n} .
\end{eqnarray*}
We have shown that the sequence $(x_{r(n+1)} + y_{r(n+1)} - x_{r(n)} - y_{r(n)})_n$
converges computably to $0$.
\begin{comment}
%%%
Let $(x_n)_n$ be a nearly computably Cauchy sequence.
Let $\lambda$ be a real number.
We claim that the sequence $(\lambda \cdot x_n)_n$ is nearly computably Cauchy as well.
Let $r:\IN\to\IN$ be a computable and strictly increasing function.
We claim that the sequence $(\lambda \cdot x_{r(n+1)} - \lambda \cdot x_{r(n)})_n$
converges computably to $0$.
It is clear that it converges to $0$ because the sequence $(x_{r(n+1)} - x_{r(n)})_n$
converges computably to $0$.
Let $g:\IN\to\IN$ be a computable modulus of convergence of this sequence.
Let $k\in\IN$ be a number with $2^k>|\lambda|$.
Then the function $h:\IN\to\IN$ defined by $h(n):=g(n+k)$, for $n\in\IN$ is computable.
It is a modulus of convergence
of the sequence $(\lambda \cdot x_{r(n+1)} - \lambda \cdot x_{r(n)})_n$ because,
for all $m,n\in\IN$ with $m\geq h(n)$,
\[ 
|\lambda \cdot x_{r(m+1)} - \lambda \cdot x_{r(m)}|
 \leq  2^k \cdot |x_{r(m+1)} - x_{r(m)}| \\
 \leq  2^k \cdot 2^{-(n+k)} \\
 = 2^{-n} .
\]
%%%
\end{comment}
\end{proof}

\begin{proposition}
\label{proposition:ncc-seq-all}
Let $(x_n)_n$ and $(y_n)_n$ be convergent nondecreasing computable sequences
of real numbers with the same limit.
Then $(x_n)_n$ is nearly computably Cauchy if, and only if,
$(y_n)_n$ is nearly computably Cauchy.
\end{proposition}

\begin{proof}
By symmetry it is sufficient to show one implication.
So, let us assume that $(x_n)_n$ is nearly computably Cauchy.
We wish to show that $(y_n)_n$ is nearly computably Cauchy.
First we show that we can switch from $(x_n)_n$ and $(y_n)_n$
to strictly increasing computable sequences of rational numbers.
There exists a computable sequence $(q_n)_n$ of rational numbers such that, for all $k,n\in\IN$,
$|x_n - q_{\langle n,k\rangle}| \leq 2^{-k}$.
We define a sequence $(a_n)_n$ of rational numbers by
$a_n:=q_{\langle n,n+3\rangle} - 5 \cdot 2^{-(n+3)}$, for $n\in\IN$.
It is clear that $(a_n)_n$ is a computable sequence of rational numbers with
$0 < 2^{-(n+1)} \leq x_n - a_n \leq 3\cdot 2^{-(n+2)}$, for all $n$. It is strictly increasing because
\begin{eqnarray*}
 a_n &=& q_{\langle n,n+3\rangle} - 5 \cdot 2^{-(n+3)} \\
   &\leq & x_n + 2^{-(n+3)} - 5 \cdot 2^{-(n+3)} \\
   &=& x_n - 2^{-(n+1)} \\
   &\leq & x_{n+1} - 2^{-(n+1)} \\
   &<& x_{n+1} - 2^{-(n+4)} - 5 \cdot 2^{-(n+4)} \\
   &\leq & q_{\langle n+1,n+4\rangle} - 5 \cdot 2^{-(n+4)} \\
   &=& a_{n+1} ,
\end{eqnarray*}
for all $n$.
In the same way we can conclude that there exists a computable strictly increasing
sequence $(b_n)_n$ of rational numbers with
$0 < 2^{-(n+1)} \leq y_n - b_n \leq 3\cdot 2^{-(n+2)}$, for all $n$. 

According to Lemma~\ref{lemma:ncc-pair-close} the sequence $(a_n)_n$ is nearly computably Cauchy.
The main step is to show that this implies that $(b_n)_n$ is nearly computably Cauchy as well.
Once that is done, by Lemma~\ref{lemma:ncc-pair-close}
the sequence $(y_n)_n$ is nearly computably Cauchy as well.

So, we consider two strictly increasing computable sequences $(a_n)_n$ and $(b_n)_n$ of rational
numbers converging to the same limit such that $(a_n)_n$ is nearly computably Cauchy,
and we wish to show that $(b_n)_n$ is nearly computably Cauchy as well.
Let $r:\IN\to\IN$ be an arbitrary strictly increasing computable function.
It is sufficient to show that the sequence
$(b_{r(n+1)} - b_{r(n)})_n$ converges computably to $0$.
The sequence $(c_n)_n$ defined by $c_n:=b_{r(n)}$, for $n\in\IN$, is computable, strictly increasing
and converges to the same number as $(a_n)_n$.
It is sufficient to show that the sequence $(c_{n+1}-c_n)_n$ converges computably to $0$.
In order to show this we define two strictly increasing computable functions $s,t:\IN\to\IN$ so that the
sequences $(a_{s(n)})_n$ and $(c_{t(n)})_n$ overtake each other in turns.
Let $s,t:\IN\to\IN$ be defined by $s(0):=0$, by $t(0) := \min\{k\in\IN ~:~ c_k>a_0\}$, and, for $n\in\IN$, by
\begin{eqnarray*}
 s(n+1) &:=& \min\{k \in\IN ~:~ k>s(n) \text{ and } a_k> c_{t(n)}\}, \\
 t(n+1) &:=& \min\{k \in\IN ~:~ k>t(n) \text{ and } c_k> a_{s(n+1)}\}. 
\end{eqnarray*}
It is clear that both functions $s$ and $t$ are strictly increasing and computable and satisfy
\[ a_{s(n)} < c_{t(n)} < a_{s(n+1)} , \]
for all $n\in\IN$.
As $(a_n)_n$ is nearly computably Cauchy the sequence
$(a_{s(n+1)}-a_{s(n)})_n$ converges computably to $0$.
Let $g:\IN\to\IN$ be a modulus of convergence of this sequence.
We claim that the computable function $h:\IN\to\IN$ defined by
$h(n):=g(n+1)$, for $n\in\IN$, is a modulus of convergence of the sequence
$(c_{t(n+1)}-c_{t(n)})_n$ and that this sequence converges to $0$.
Indeed, for $m,n\in\IN$ with $m\geq h(n)$ we obtain
\begin{eqnarray*}
|c_{t(m+1)}-c_{t(m)}|
&=& c_{t(m+1)}-c_{t(m)} \\
&< & a_{s(m+2)} - a_{s(m)} \\
&=& a_{s(m+2)} - a_{s(m+1)} + a_{s(m+1)} - a_{s(m)} \\
&\leq & 2^{-(n+1)} + 2^{-(n+1)} \\
&=& 2^{-n} .
\end{eqnarray*}
We have shown that the sequence $(c_{t(n+1)}-c_{t(n)})_n$ converges computably to $0$.
By Lemma~\ref{lemma:compconv-h1} the sequence $(c_{n+1} - c_n)_n$ converges computably to $0$ as well.
\end{proof}

The following corollary of Proposition~\ref{proposition:ncc-seq-all}
shows that the converse of Lemma~\ref{lemma:nearly-computably-convergent-1}
is true for nondecreasing computable sequences of real numbers.

\begin{corollary}
\label{corollary:nearly-computably-convergent-2}
Let $(x_n)_n$ be a nondecreasing computable sequence of real numbers.
Let $t:\IN\to\IN$ be a total, nondecreasing, unbounded computable function, and define
$y_n:=x_{t(n)}$.
Then $(x_n)_n$ is nearly computably Cauchy if, and only if,
$(y_n)_n$ is nearly computably Cauchy.
\end{corollary}

\section{The Field of Nearly Computable Real Numbers}
\label{section:nearly-computable-field}

In this section we consider the set of nearly computable real numbers and their closure properties.
Actually, ``nearly computable'' elements can be defined in any computable metric space.

\begin{definition}
Let $(X,d,\alpha)$ be a computable metric space.
We call an element $y\in X$ {\em nearly computable} if there exists a computable sequence
$x:\IN\to X$ in $X$ that is nearly computably Cauchy and converges to $y$.
\end{definition}

In this definition, $\alpha$ can be replaced by any
equivalent (see the end of Section~\ref{section:comp-sequences})
sequence $\beta$ in $X$ without changing the set of nearly computable elements of $X$.

\begin{proposition}
\label{prop:comp-nearly-comp}
Let $(X,d,\alpha)$ be a computable metric space.
Every computable element of $X$ is nearly computable.
\end{proposition}

\begin{proof}
This follows from 
Lemma~\ref{lemma:comp-metric-space-comp-elem},
Corollary~\ref{cor:conv-comp-Cauchy}.\ref{cor:conv-comp-Cauchy-1},
and Lemma~\ref{lemma:cChauchy-ncc-1}.
\end{proof}

By applying this definition to the computable metric space $(\IR,d,\nu_\IQ)$ where $d$ is the usual Euclidean
distance, we obtain the notion of a {\em nearly computable real number}.
Let $\NC$ be the set of all nearly computable real numbers.

We say that a subset $S\subseteq \IR$ {\em is closed under} a function $f:\subseteq\IR^k\to\IR$
if for any vector $x\in\dom(f)\cap S^k$ the real number $f(x)$ is an element of $S$.
For $k\geq 1$ and for $i,j\in\IN$ we write
\[ B^{(k)}_{\langle i,j\rangle} := \left\{x \in\IR^k ~:~ d(x,\nu_\IQ^{(k)}(i)) < 2^{-j}\right\} . \]
The following notion of a computable function $f:\subseteq \IR^{k}\to\IR$
can be characterized in many different ways; compare~\cite{Wei00}.

\begin{definition}
\label{definition:computable-function}
A function $f:\subseteq \IR^{k}\to\IR$ is called {\em computable}
if there exists a computably enumerable subset $A \subseteq\IN$ such that
\begin{enumerate}
\item
for all $m,n\in\IN$ with $\langle m,n\rangle \in A$, $f( B^{(k)}_m)\subseteq B^{(1)}_n$,
\item
for all $x\in\dom(f)$ and all $\varepsilon>0$ there exist $m,n\in\IN$
with $x \in B^{(k)}_m$, $\langle m,n\rangle \in A$, and $2^{-\pi_2(n)}<\varepsilon$.
\end{enumerate}
\end{definition}

The following characterization of computable functions on a compact rectangle with rational vertices
is a minor straightforward generalization of a result by Grzegorczyk~\cite{Grz57}. We omit the proof.

\begin{lemma}
\label{lemma:Grze}
Let $k\geq 1$. Let $a_1,\ldots,a_k,b_1,\ldots,b_k$ be rational numbers with $a_i< b_i$, for $i=1,\ldots,k$.
Then a function $f:[a_1,b_1]\times \ldots \times [a_k,b_k]\to\IR$ is computable if, and only if, 
\begin{enumerate}
\item
for any computable sequence $(q_n)_n$ of rational points in $(a_1,b_1)\times \ldots \times (a_k,b_k)$ 
(that is, for any sequence $(q_n)_n \in \IQ^k \cap (a_1,b_1)\times \ldots \times (a_k,b_k)$ 
such that there is a computable function $f:\IN\to\IN$ with $q_n=\nu_\IQ^{(k)}(f(n))$, for all $n$)
the sequence $(f(q_n))_n$ is a computable sequence of real numbers,
\item
there exists a computable function $g:\IN\to\IN$ such that,
for all $n\in\IN$ and all $x,y \in [a_1,b_1]\times \ldots \times [a_k,b_k]$,
if $d(x,y)\leq 2^{-g(n)}$ then $|f(x) - f(y)| \leq 2^{-n}$
(such a function $g$ will be called a {\em computable modulus of continuity of $f$}).
\end{enumerate}
\end{lemma}

\begin{proposition}
\label{proposition:nc-closed-under-functions}
The set $\NC$ is closed under computable functions $f:\subseteq\IR^k\to\IR$
with open domain $\dom(f)\subseteq\IR^k$, for arbitrary $k>0$.
\end{proposition}

\begin{proof}
Let $x_1,\ldots,x_k$ be nearly computable numbers, and let $f:\subseteq\IR^k\to\IR$
be a computable function with open domain and with $(x_1,\ldots,x_k)\in\dom(f)$.
In the following we simply write $x$ for the vector $(x_1,\ldots,x_k)$.
We wish to show that $f(x)$ is a nearly computable number.
Let $a_1,\ldots,a_k,b_1,\ldots,b_k$ be rational numbers with $a_i< x_i< b_i$, for $i=1,\ldots,k$,
and with $[a_1,b_1]\times \ldots \times [a_k,b_k] \subseteq \dom(f)$.
The restriction of $f$ to the cube $[a_1,b_1]\times \ldots \times [a_k,b_k]$ is computable as well.
Let $g:\IN\to\IN$ be a computable modulus of continuity of this restriction; compare Lemma~\ref{lemma:Grze}.
Let $(q_n)_n$ be a computable sequence of rational points that converges to $x$
and is nearly computably Cauchy.
We can assume that $q_n\in (a_1,b_1)\times \ldots \times (a_k,b_k)$, for all $n$.
If this should not be the case then we can achieve it by cutting off a finite prefix
of the sequence $(q_n)_n$ so that the resulting sequence satisfies this condition.
The resulting sequence will still be computable, will still converge to $x$
and will still be nearly computably Cauchy.
By Lemma~\ref{lemma:Grze} the sequence $(f(q_n))_n$ is a computable sequence of real numbers
that converges to $f(x)$. 
We claim that it is nearly computably Cauchy.
Let $r:\IN\to\IN$ be a strictly increasing, computable function.
Let $h:\IN\to\IN$ be a computable modulus of convergence of the sequence $d(q_{r(n+1)},q_{r(n)})$.
We claim that the computable function $h\circ g$ is a modulus of convergence of the sequence
$|f(q_{r(n+1)}) - f(q_{r(n)})|$. Indeed, for $m,n\in\IN$ with $m\geq h(g(n))$ we obtain
$d(q_{r(m+1)},q_{r(m)}) \leq 2^{-g(n)}$, hence
$|f(q_{r(m+1)}) - f(q_{r(m)})| \leq 2^{-n}$.
Thus, the sequence $(f(q_n))_n$ is a computable sequence of real numbers
that converges to $f(x)$ and is nearly computably Cauchy.
Finally, let $(p_n)_n$ be a computable sequence of rational numbers with 
$|f(q_n)-p_n|\leq 2^{-n}$, for all $n$.
Then this sequence converges to $f(x)$ as well. By Lemma~\ref{lemma:ncc-pair-close}
it is nearly computably Cauchy as well.
This shows that $f(x)$ is a nearly computable number.
\end{proof}

A subset $S\subseteq\IR$ is called a {\em subfield} of $\IR$ if $0,1\in S$ and if $S$ is closed under the
three total functions addition $+:\IR^2\to\IR$, additive inverse $-:\IR\to\IR$,
and multiplication $\cdot:\IR^2\to\IR$
and under the partial function multiplicative inverse $1/\cdot:\IR\setminus\{0\}\to\IR$.
A subset $S\subseteq\IR$ is called a {\em real closed field} if it is a subfield of $\IR$ and additionally
for every $k>0$ and for every nonzero vector of real numbers $b_0,\ldots,b_k \in S$
(``nonzero vector'' means that at least one of the numbers $b_0,\ldots,b_k$ must be different from zero)
every real number $x$ with $b_k \cdot x^k + \ldots + b_1 \cdot x + b_0=0$ is an element of $S$,
that is, any real zero of any nonzero polynomial with coefficients in $S$ is again an element of $S$.

Let $(X,d_X)$ and $(Y,d_Y)$ be metric spaces. A function $f:X\to Y$ is called {\em Lipschitz continuous}
if there exists a constant $c>0$ such that, for all $x,x'\in X$,
\[ d_Y( f(x), f(x') ) \leq c \cdot d_X(x,x') . \]
Such a constant $c$ is then called a {\em Lipschitz constant} of $f$.
It is well known that all norms on $\IR^k$ are equivalent.
Therefore, a function $f:\subseteq\IR^k\to\IR$ is Lipschitz continuous
with respect to the distance function induced by one norm if, and only if,
it is Lipschitz continuous with respect to the distance function induced by any other norm.
The following lemma is a slight strengthening of an observation
implicitly contained in the article~\cite{Rai05b} by Raichev.

\begin{lemma}
\label{lemma:real-closed-field}
If a subset $K\subseteq\IR$ contains a number $x_0\neq 0$ and is closed
under Lipschitz continuous computable functions
$f:\subseteq\IR^k\to\IR$ with open domain $\dom(f)\subseteq\IR^k$, $k$ arbitrary, then
$K$ is a real closed subfield of $\IR$.
\end{lemma}

\begin{proof}
On $\IR^k$ let us consider the maximum norm $||\cdot||_\infty:\IR^k\to\IR^k$ defined by
\[ ||(x_1,\ldots,x_k)||_\infty := \max\{ |x_1|,\ldots,|x_k|\} , \]
for any vector $(x_1,\ldots,x_k)\in\IR^k$.

The additive inverse $-:\IR\to\IR$ is a total, Lipschitz continuous, computable function
with Lipschitz constant $1$.
The addition $+:\IR^2\to\IR$ is a total, Lipschitz continuous, computable function with Lipschitz constant $2$.
For any two numbers $x,x'\in\IR$ let us choose an integer $N$ with $N>|x|$ and $N>|x'|$.
The restriction $\cdot|_{(-N,N)^2}:(-N,N)^2\to\IR$ of the multiplication function to the open cube $(-N,N)^2$ is 
a Lipschitz continuous and computable function.
For any number $x\in\IR\setminus\{0\}$ let us choose an integer $N>1$ with $N>|x|$ and $1/N<|x|$.
The restriction $\cdot|_{(-N,-1/N) \cup (1/N,N)}:(-N,-1/N) \cup (1/N,N)\to\IR$
of the multiplicative inverse function to the
union of the two open intervals $(-N,-1/N)$ and $(1/N,N)$ is a Lipschitz continuous and computable function.
This shows that $K$ is closed under the
three total functions addition $+:\IR^2\to\IR$, additive inverse $-:\IR\to\IR$,
and multiplication $\cdot:\IR^2\to\IR$
and under the partial function multiplicative inverse $1/\cdot:\IR\setminus\{0\}\to\IR$.
Let $x_0$ be an element of $K$ with $x_0\neq 0$.
Then $0=x_0 + (-x_0)\in K$ and $1 = x_0 \cdot (1/x_0) \in K$.
We have shown that $K$ is a subfield of $\IR$.

Let $z\in K$ be a real number such that there exists a nonzero vector of real numbers $b_0,\ldots,b_k \in K$
with $b_k \cdot z^k + \ldots + b_1 \cdot z + b_0=0$. 
Then $k>0$.
We wish to show that $z$ is an element of $K$ as well.
Let $p(x)$ be the polynomial $p(x):=\sum_{i=0}^k b_i \cdot x^i$.
In the case $p(z)=p'(z)=0$ we can switch from the polynomial $p(x)$ to the polynomial $p'(x)$.
Note that in this case the coefficients of $p'(x)$ are elements of $K$ as well,
the vector of these coefficients is still nonzero, and the degree of $p'(x)$
is smaller than the degree of $p(x)$. As we can repeat this, we can assume that
$p(x)=\sum_{i=0}^k b_i\cdot x^i$ is a polynomial with a nonzero vector of coefficients $b_0,\ldots,b_k\in K$,
with $p(z)=0$ (note that this implies $k>0$), and with $p'(z)\neq 0$.
The function $F:\IR^{k+2}\to\IR$ defined by
\[ F(x):= \sum_{i=0}^k x_{i+1} \cdot x_{k+2}^i , \]
for $x=(x_1,\ldots,x_{k+2})\in\IR^{k+2}$
is computable and continuously differentiable and has computable partial derivatives.
It satisfies $F(b_0,\ldots,b_k,z)=0$ and $\frac{\partial F}{\partial x_{k+2}}(b_0,\ldots,b_k,z) \neq 0$.
By the effective version~\cite[Theorem 4.3]{McN08a} of the Implicit Function Theorem due to McNicholl
(and by the fact that, given computable input, a computable functional produces computable output)
there exist an open subset $V\subseteq\IR^{k+1}$ 
and a computable and continuously differentiable function $g:V\to\IR$ such that $(b_0,\ldots,b_k)\in V$,
$g(b_0,\ldots,b_k)=z$, and, for all $x=(x_1,\ldots,x_{k+1})\in V$, $\sum_{i=0}^k x_{i+1} \cdot g(x)^i = 0$.
As any continuously differentiable function is locally Lipschitz continuous
there exists an open subset $W\subseteq V$
with $(b_0,\ldots,b_k)\in W$ such that the restriction $g|_W$ is Lipschitz continuous.
This restriction $g|_W$ is computable as well.
Hence, the number $z$ is an element of $K$.
\end{proof}

\begin{corollary}
\label{cor:NC-real-closed-field}
The set $\NC$ is a real closed field.
\end{corollary}

\begin{proof}
This follows from Lemma~\ref{lemma:real-closed-field} and
Proposition~\ref{proposition:nc-closed-under-functions}
and the fact that $1\in \NC$.
\end{proof}

We have seen in Proposition~\ref{proposition:nc-closed-under-functions}
that the set $\NC$ is closed under computable real functions with open domain.
Is it perhaps even closed under arbitary computable real functions? 
A negative answer is given by the following theorem and by the fact that
$\NC \setminus \EC$ is not empty; see Theorem~\ref{theorem:downey-laforte} in the following section.

\begin{theorem}
\label{theorem:range-comp-func}
There exists a computable (and strictly increasing) function $f:\IR\setminus\IQ\to\IR$
such that for every noncomputable
real number $\alpha$ the number $f(\alpha)$ is not nearly computable.
\end{theorem}

\begin{corollary}
\label{corollary:nc-not-closed-under-comp-functions}
The set $\NC$ of nearly computable numbers is not closed under arbitrary computable
functions $f:\subseteq\IR\to\IR$.
\end{corollary}

\begin{proof}
A function $f$ as in Theorem~\ref{theorem:range-comp-func} maps any element
of $\NC\setminus\EC$ to $\IR\setminus \NC$.
As it is known that $\NC\setminus\EC$ is not empty (see Theorem~\ref{theorem:downey-laforte}),
we conclude that $\NC$ is not closed under $f$.
\end{proof}

In the proof of Theorem~\ref{theorem:range-comp-func} we use the following proposition.

\begin{proposition}
\label{proposition:4-A}
For a set $A\subseteq\IN$ the following conditions are equivalent.
\begin{enumerate}
\item
$A$ is decidable.
\item
The number $4^{-A}:=\sum_{n\in A} 4^{-n}$ is computable.
\item
The number $4^{-A}$ is nearly computable.
\end{enumerate}
\end{proposition}

\begin{proof}
The implication ``1 $\Rightarrow$ 2'' is obvious.
The implication ``2 $\Rightarrow$ 3'' follows from Corollary~\ref{cor:conv-comp-Cauchy}
and Lemma~\ref{lemma:cChauchy-ncc-1}.
We show ``3 $\Rightarrow$ 1''.
Let us consider a set $A \subseteq \IN$ such that the number $4^{-A}$ is nearly computable.
We can assume without loss of generality $A\neq \emptyset$, hence, $4^{-A}>0$.
Let $(a_n)_n$ be a computable sequence of positive rational numbers
that is nearly computably Cauchy and converges to $4^{-A}$.
We define a function $r:\IN\to\IN$ by $r(-1):=-1$ and
\[ r(n) := \min\left\{ k \in \IN ~:~ k> r(n-1) \text{ and } 
      (\exists B\subseteq \{0,\ldots,n\}) \left|4^{-B} - a_k\right| < \frac{1}{2}\cdot 4^{-n} \right\}. \]
The function $r$ is well-defined because the sequence $(a_k)_k$ converges to $4^{-A}$, and we see
\[ |4^{- (A \cap \{0,\ldots,n\})} - 4^{-A}| \leq \sum_{i=n+1}^\infty 4^{-i} = \frac{1}{3} \cdot 4^{-n}
   < \frac{1}{2}\cdot 4^{-n} . \]
It is clear that $r$ is strictly increasing and computable.
Actually, for every $n\in\IN$ there is exactly one set $B\subseteq \{0,\ldots,n\}$
with $|4^{-B} - a_{r(n)}| < \frac{1}{2}\cdot 4^{-n}$
because the numbers $4^{-C}$ for $C\subseteq \{0,\ldots,n\}$ have distance at least $4^{-n}$ from each other.
Let us call this set $B_n$. Note that given $n$ we can compute $B_n$.
As $(a_n)_n$ is nearly computably Cauchy the sequence $(a_{r(n+1)} - a_{r(n)})_n$
converges computably to $0$.
Let $g:\IN\to\IN$ be a strictly increasing, computable modulus of convergence of this sequence.
We claim that 
\begin{equation}
\label{eq:claim-4-A}
(\forall m \geq g(2n+3)) B_m \cap \{0,1,\ldots,n\} = A \cap \{0,1,\ldots,n\}.
\end{equation}
Once this is shown it is clear that $A$ is decidable because given $n$ we can compute $B_{g(2n+3)}$
and can check whether $n\in B_{g(2n+3)}$ or not.
Let us show the claim~\eqref{eq:claim-4-A}.
For every $m\geq n$ the set $B_m \cap\{0,1,\ldots,n\}$ satisfies
\begin{eqnarray*}
|4^{-(B_m \cap\{0,1,\ldots,n\})} - a_{r(m)}| 
   & \leq & |4^{-(B_m \cap\{0,1,\ldots,n\})} - 4^{-B_m}| + |4^{-B_m} - a_{r(m)}| \\
   & < & \sum_{i=n+1}^m 4^{-i} + \frac{1}{2}\cdot 4^{-m}
   = \frac{1}{3}\cdot 4^{-n} - \frac{1}{3} \cdot 4^{-m} + \frac{1}{2}\cdot 4^{-m} \\
   &=& \frac{1}{3}\cdot 4^{-n} + \frac{1}{6} \cdot 4^{-m} \\
   &\leq & \frac{1}{2} \cdot 4^{-n} ,
\end{eqnarray*}
and it is the only set $C\subseteq\{0,1,\ldots,n\}$ with $|4^{-C} - a_{r(m)}| < \frac{1}{2} \cdot 4^{-n}$
because the numbers $4^{-C}$ for $C\subseteq \{0,\ldots,n\}$ have distance at least $4^{-n}$ from each other.
By reusing part of the calculation above we obtain for $m\geq g(2n+3)$
\begin{eqnarray*}
|4^{-(B_{m+1} \cap \{0,1,\ldots,n\})} - a_{r(m)}|
&\leq & |4^{-(B_{m+1} \cap \{0,1,\ldots,n\})} - a_{r(m+1)}| + |a_{r(m+1)} - a_{r(m)}| \\
&< & \frac{1}{3} \cdot 4^{-n} + \frac{1}{6} \cdot 4^{-(m+1)} + 2^{-(2n+3)} \\
&=& 4^{-n} \cdot \left( \frac{1}{3} + \frac{1}{6} \cdot 4^{-(m+1-n)} + \frac{1}{8} \right) \\
&<&  \frac{1}{2} \cdot 4^{-n} 
\end{eqnarray*}
(for the last step note that $m\geq g(2n+3) \geq 2n+3$ because $g$ is assumed to be strictly increasing).
So, $B_{m+1} \cap \{0,1,\ldots,n\}$ is also a set $C\subseteq \{0,1,\ldots,n\}$
with $|4^{-C} - a_{r(m)}| < \frac{1}{2} \cdot 4^{-n}$.
But we just saw that there can be only one set $C$ with this property,
and that that is the set $B_m\cap\{0,1,\ldots,n\}$.
We conclude
\begin{equation}
\label{eq:claim-4-A-b}
(\forall m \geq g(2n+3)) B_m \cap \{0,1,\ldots,n\} = B_{m+1} \cap \{0,1,\ldots,n\}.
\end{equation}
For $m$ so large that $|4^{-A} - a_{r(m)}| < \frac{1}{6} \cdot 4^{-n}$ we similarly observe that
\[ |4^{-(A \cap \{0,1,\ldots,n\})} - a_{r(m)}| 
\leq |4^{-(A \cap \{0,1,\ldots,n\})} - 4^{-A}| + |4^{-A} - a_{r(m)}| 
   < \sum_{i=n+1}^\infty 4^{-i} + \frac{1}{6}\cdot 4^{-n} 
   = \frac{1}{2} \cdot 4^{-n} . \]
By the same uniqueness argument as above we conclude that for sufficiently large $m$ we have
$B_m\cap\{0,1,\ldots,n\}=A\cap\{0,1,\ldots,n\}$.
This and~\eqref{eq:claim-4-A-b} imply \eqref{eq:claim-4-A}.
This ends the proof of the implication ``3 $\Rightarrow$ 1''.
\end{proof}

\begin{proof}[{Proof of Theorem~\ref{theorem:range-comp-func}}]
We define a function $f:\IR\setminus\IQ\to\IR$ by
\[ f(\alpha) := \sum_{n\colon \nu_\IQ(n)<\alpha} 4^{-n} , \]
for $\alpha\in\IR\setminus\IQ$.
This function is computable because, for any $n$ and any irrational number $\alpha$,
if one knows $\alpha$ with sufficient precision
then one can decide whether $\nu_\IQ(n)<\alpha$ or not. 
It is clear that the function $f$ is strictly increasing.

Let us consider some irrational real number $\alpha$ such that $f(\alpha)$ is nearly computable.
Then by Proposition~\ref{proposition:4-A} the set $\{n \in \IN ~:~ \nu_\IQ(n) < \alpha\}$ is decidable.
Hence, $\alpha$ is computable.
\end{proof}

\section{Left computable, nearly computable numbers}
\label{section:nearly-computable-left-computable}

In this section we are going to have a look at the set $\LC\cap\NC$ of real numbers that are at the same time
left-computable and nearly computable.

A subset $A\subseteq\{0,1\}^*$ is called {\em prefix-free} 
if, for all $u,v\in A$, if $u\subseteq v$ then $u=v$, in other words,
if no string $u$ in $A$ is a proper prefix of any other string $v$ in $A$.
We shall make use of the following well-known effective version of the Kraft theorem,
sometimes called Kraft-Chaitin Theorem; see, e.g.,~\cite[Theorem 3.6.1]{DH10}.

\begin{lemma}
\label{lemma:Kraft-Chaitin}
Let $f:\IN\to\IN$ be a computable function such that $\sum_{n=0}^\infty 2^{-f(n)}\leq 1$.
Then there exists an injective computable function $h:\IN\to\{0,1\}^*$ such that the set $\{h(n)~:~n\in\IN\}$
is prefix-free and $|h(n)|=f(n)$, for all $n\in\IN$.
\end{lemma}

The equivalence of the conditions \eqref{proposition:nearly-comp-eq-4}
and \eqref{proposition:nearly-comp-eq-6}
in the following proposition is essentially the content of Theorem 5 in \cite{SW05}.

\begin{proposition}
\label{proposition:nearly-comp-eq}
For a left-computable number $\alpha$ the following four conditions are equivalent:
\begin{enumerate}
\item
\label{proposition:nearly-comp-eq-1}
It is nearly computable.
%, that is, there exists a computable and nearly computably Cauchy sequence $(a_n)_n$ of rational numbers
%converging to $\alpha$.
\item
\label{proposition:nearly-comp-eq-2}
There exists a strictly increasing, computable, and nearly computably Cauchy sequence $(a_n)_n$
of rational numbers converging to $\alpha$.
\item
\label{proposition:nearly-comp-eq-3}
Every nondecreasing, computable sequence $(a_n)_n$ of real numbers
converging to $\alpha$ is nearly computably Cauchy.
\item
\label{proposition:nearly-comp-eq-4}
For every strictly increasing, computable sequence $(a_n)_n$ of rational numbers
converging to $\alpha$ the sequence $(a_{n+1}-a_n)_n$ converges computably to $0$.
\end{enumerate}
For a left-computable number $\alpha$ with $\alpha\geq 0$ the following condition is equivalent to these conditions as well:
\begin{enumerate}
\setcounter{enumi}{4}
\item
\label{proposition:nearly-comp-eq-5}
Every computable sequence $(b_n)_n$ of positive rational numbers with $\sum_{i=0}^\infty b_i=\alpha$
converges computably to $0$.
\end{enumerate}
For a left-computable number $\alpha \in [0,1]$ the following condition is equivalent to these conditions as well:
\begin{enumerate}
\setcounter{enumi}{5}
\item
\label{proposition:nearly-comp-eq-6}
Every prefix-free, computably enumerable set $A\subseteq \{0,1\}^*$ with
$\alpha=\sum_{w \in A} 2^{-|w|}$ is decidable.
\end{enumerate}
\end{proposition}

\begin{proof}[Proof of Proposition~\ref{proposition:nearly-comp-eq}]
``\eqref{proposition:nearly-comp-eq-1} $\Rightarrow$ \eqref{proposition:nearly-comp-eq-2}'':
Let $(a_n)_n$ be a computable and nearly computably Cauchy sequence of rational numbers
converging to $\alpha$.
Let $(b_n)_n$ be a computable and strictly increasing sequence of rational numbers converging to $\alpha$.
As both sequences have the same limit and are computable there exists
a strictly increasing, computable function $s:\IN\to\IN$ such that,
for all $n$, $|a_{s(n)} - b_{s(n)}|\leq 2^{-n}$.
By Lemma~\ref{lemma:nearly-computably-convergent-1}
the sequence $(a_{s(n)})_n$ is nearly computably Cauchy as well,
and it is clear that it converges to $\alpha$ as well.
By Lemma~\ref{lemma:ncc-pair-close} the sequence $(b_{s(n)})_n$ is nearly computably Cauchy as well,
and it is clear that it converges to $\alpha$ as well.
Furthermore, this sequence is strictly increasing.

``\eqref{proposition:nearly-comp-eq-2} $\Rightarrow$ \eqref{proposition:nearly-comp-eq-1}'':
Trivial.

``\eqref{proposition:nearly-comp-eq-2} $\Leftrightarrow$ \eqref{proposition:nearly-comp-eq-3}'':
By Proposition~\ref{proposition:ncc-seq-all}.

``\eqref{proposition:nearly-comp-eq-3} $\Rightarrow$ \eqref{proposition:nearly-comp-eq-4}'':
Trivial.

``\eqref{proposition:nearly-comp-eq-4} $\Rightarrow$ \eqref{proposition:nearly-comp-eq-2}'':
Let $(a_n)_n$ be a strictly increasing, computable sequence in $\IQ$ converging to $\alpha$.
Let $r:\IN\to\IN$ be a total, strictly increasing, computable function.
We define $b_n:=a_{r(n)}$.
Then the sequence $(b_n)_n$ is a strictly increasing, computable sequence in $\IQ$
converging to $\alpha$ as well.
By \eqref{proposition:nearly-comp-eq-4} the sequence
$(a_{r(n+1)} - a_{r(n)})_n = (b_{n+1}-b_n)_n$ converges computably to $0$. 
Hence, $(a_n)_n$ is nearly computably Cauchy.

In the rest of the proof we assume that $\alpha\geq 0$.

``\eqref{proposition:nearly-comp-eq-4} $\Rightarrow$ \eqref{proposition:nearly-comp-eq-5}'':
If $\alpha=0$ then there is no sequence $(b_n)_n$ of positive rational numbers
with $\sum_{i=0}^\infty b_i=\alpha$.
Hence, in this case nothing needs to be shown.
Let us consider the case $\alpha>0$.
Let $(b_n)_n$ be a computable sequence of positive rational numbers with $\sum_{i=0}^\infty b_i=\alpha$.
We define $a_n:=\sum_{i=0}^{n-1} b_i$, for all $n$.
Then the sequence $(a_n)_n$ is a strictly increasing, computable sequence of rational numbers
converging to $\alpha$.
By \eqref{proposition:nearly-comp-eq-4} the sequence $(a_{n+1} - a_n)_n = (b_n)_n$
converges computably to $0$.

``\eqref{proposition:nearly-comp-eq-5} $\Rightarrow$ \eqref{proposition:nearly-comp-eq-4}'':
Let us first consider the case $\alpha=0$. Then $\alpha$ is a computable number.
Let us consider a strictly increasing, computable sequence $(a_n)_n$ of rational numbers
converging to $0$. By Lemma~\ref{lemma:nondecreasing-cc}
it converges computably, by Corollary~\ref {cor:conv-comp-Cauchy}
it is computably Cauchy, by Lemma~\ref{lemma:cChauchy-ncc-1} it is nearly computably Cauchy.
Hence, in particular the sequence $(a_{n+1} - a_n)_n$ converges computably to $0$.
Now let us consider the case $\alpha>0$.
Let us consider a strictly increasing, computable sequence $(a_n)_n$ of rational numbers
converging to $\alpha$. There exists some $N>0$ with $a_N>0$. Let the sequence $(c_n)_n$ be defined
by $c_0:=0$ and $c_n:=a_{N+n}$, for $n>0$.
Then the sequence $(c_n)_n$ is a strictly increasing computable sequence of rational
numbers converging to $\alpha$. Let the sequence $(b_n)_n$ be defined by $b_n:=c_{n+1} - c_n$.
Then $(b_n)_n$ is a computable sequence of positive rational numbers with $\sum_{i=0}^\infty b_i=\alpha$.
By \eqref{proposition:nearly-comp-eq-5} it converges computably to $0$.
That means, the sequence $(c_{n+1} - c_n)_n$ converges computably to $0$.
As the sequences $(c_{n+1} - c_n)_n$ and $(a_{n+1} - a_n)_n$ are identical up to a finite prefix,
the sequence $(a_{n+1} - a_n)_n$ converges computably to $0$ as well.

In the rest of the proof we assume that $\alpha\in [0,1]$.

``\eqref{proposition:nearly-comp-eq-5} $\Rightarrow$ \eqref{proposition:nearly-comp-eq-6}'':
Let $A\subseteq \{0,1\}^*$ be a prefix-free, computably enumerable set with
$\alpha=\sum_{w \in A} 2^{-|w|}$.
If $A$ is finite then it is decidable. So, let us assume that $A$ is infinite. Let
$h:\IN\to\{0,1\}^*$ be a total, injective, computable function with $\mathrm{range}(h)=A$.
We define $b_n:= 2^{-|h(n)|}$, for all $n\in\IN$.
Then $(b_n)_n$ is a computable sequence of positive rational numbers with $\sum_{i=0}^\infty b_i=\alpha$.
By (\ref{proposition:nearly-comp-eq-5}) the sequence $(b_n)_n$ converges computably.
It is clear that it converges to $0$.
Hence, there exists a total, computable function $g:\IN\to\IN$ such that, for all $m,n\in\IN$,
\[ m \geq g(n) \Rightarrow | b_m | \leq 2^{-n} . \]
Hence, for all $m,n\in\IN$,
\[ m \geq g(n) \Rightarrow |h(m)| \geq n . \]
But this shows that $A$ is decidable: A string $w$ of length $n$ is an element of $A$
if, and only if, it is an element of $\{h(0),\ldots,h(g(n+1)-1)\}$.

``\eqref{proposition:nearly-comp-eq-6} $\Rightarrow$ \eqref{proposition:nearly-comp-eq-5}'':
In the case $\alpha=0$ nothing needs to be shown. So, we consider the case $0 < \alpha \leq 1$.
We assume that (\ref{proposition:nearly-comp-eq-6}) is true.
Let $(b_n)_n$ be a computable sequence of positive rational numbers with $\sum_{i=0}^\infty b_i = \alpha$.
We wish to show that it converges computably. 
We define a function $f:\IN\to\IN$ recursively as follows:
\[ 
 f(0) := \min\left\{k \in\IN ~:~ 2^{-k} < b_0\right\},
\]
and, assuming that we have defined $f(0),\ldots,f(n)$ so that
$\sum_{i=0}^n 2^{-f(i)} < \sum_{i=0}^n b_i$, we define
\[
 f(n+1) := \min\left\{k\in\IN ~:~ \sum_{i=0}^n 2^{-f(i)} + 2^{-k} < \sum_{i=0}^{n+1} b_i \right\} .
\]
Then $f$ is total and computable and satisfies 
$\sum_{i=0}^\infty 2^{-f(i)} = \alpha$.
Furthermore, $2^{-f(n)} \geq b_n/2$, for all $n$.
By Lemma~\ref{lemma:Kraft-Chaitin} there exists a total, injective, computable function
$h:\IN\to\{0,1\}^*$ such that the set $A:=h(\IN)=\{h(n)~:~n\in\IN\}$ is prefix-free and satisfies
$|h(i)| = f(i)$, for all $i\in\IN$. The set $A$ is computably enumerable.
By (\ref{proposition:nearly-comp-eq-6}) it is decidable.
The function $g:\IN\to\IN$ defined by
\[ g(n) := \min\left\{k\in\IN ~:~ A \cap \left(\bigcup_{i=0}^n \{0,1\}^i\right)
   \subseteq \{h(0),\ldots,h(k)\} \right\} , \]
for all $n\in\IN$ is total and computable. We observe that, for all $i>g(n)$,
\[ f(i)=|h(i)| \geq n+1, \]
hence, $b_i/2\leq 2^{-f(i)} \leq 2^{-(n+1)}$.
This shows that the sequence $(b_n)_n$ converges computably.
\end{proof}

The main insight in the proof of the following theorem is due to Downey and LaForte~\cite{DL02}.

\begin{theorem}
\label{theorem:downey-laforte}
$\EC \subsetneq \LC\cap\NC$.
\end{theorem}

\begin{proof}
The inclusion $\EC\subseteq\LC$ is clear.
The inclusion $\EC\subseteq\NC$ was already stated in Proposition~\ref{prop:comp-nearly-comp}.
The fact that there exists a real number that is left-computable and nearly computable but not computable
follows from Proposition~\ref{proposition:nearly-comp-eq}
and from the main result by Downey and LaForte~\cite{DL02} who proved that there exists
a real number in the interval $[0,1]$ that is left-computable but not computable
and satisfies Condition \eqref{proposition:nearly-comp-eq-6}
in Proposition~\ref{proposition:nearly-comp-eq}.
\end{proof}

The proof of the main result of Downey and LaForte~\cite{DL02} is by an infinite injury priority argument.
Thus, one may say that the set of real numbers that are left-computable and nearly computably
is only slightly larger than the set of computable numbers.

Often the requirement that a real number should be left-computable and nearly computable
in combination with some other effectivity requirement forces the real number to be even computable.
In this section we discuss two observations of this kind
and in Section~\ref{section:nearly-computable-randomness} three further ones.

\begin{definition}
For a set $A\subseteq \IN$ let $0.A:=\sum_{n\in A} 2^{-n-1}$
\end{definition}

If one identifies a set $A\subseteq\IN$ with its characteristic function
$\chi_A\in\{0,1\}^\IN$, an infinite binary sequence,
then one can say that $0.A$ is the real number in the interval $[0,1]$ with binary representation $0.A$.
Note that the function $A\mapsto 0.A$ is a surjection from the power set
$\mathcal{P}(\IN):=\{A ~:~ A \subseteq \IN\}$ onto the real interval $[0,1]$.
It is also well known and easy to see that for any set $A\subseteq \IN$ the real number 
$0.A$ is computable if, and only if, $A$ is decidable.

\begin{definition}
A real number $\alpha$ is called {\em strongly left-computable} if there exists a computably enumerable
set $A\subseteq \IN$ such that $\alpha = 0.A$.
\end{definition}

Any computable number in $[0,1]$ is strongly left-computable.
The existence of a c.e. set $A\subseteq \IN$ that is not decidable implies that there are strongly left-computable
real numbers that are not computable. 
It is clear that any strongly left-computable number is left-computable.
But it is also well known and easy to see 
that there are left-computable numbers in the interval $[0,1]$
that are not strongly left-computable; see Soare~\cite{Soa69}.
So, any computable number in $[0,1]$ is strongly left-computable and nearly computable.
An argument by Downey, Hirschfeldt, and LaForte~\cite[Pages 105, 106]{DHL04}
shows that the converse is true as well.

\begin{proposition}[{Downey, Hirschfeldt, and LaForte~\cite[Pages 105, 106]{DHL04}}]
\label{proposition:strongly-lc}
Any strongly left-computable and nearly computable number is computable.
\end{proposition}

For completeness sake we give the easy proof.

\begin{proof}
Let us consider a computably enumerable set $A \subseteq \IN$ such that the
real number $0.A$ is nearly computable.
We claim that $0.A$ is even computable.
It is sufficient to show that $A$ is decidable.
If $A$ is finite this is clear anyway. So, let us assume that $A$ is infinite. Let
$f:\IN\to\IN$ be a total, injective, computable function with $f(\IN)=A$.
Then the sequence $(b_n)_n$ defined by $b_n:=2^{-f(n)}$, for $n\in\IN$,
is a computable sequence of positive rational numbers with $\sum_{i=0}^\infty b_i=2^{-A}$.
As we assume that $0.A$ is nearly computable, by Proposition~\ref{proposition:nearly-comp-eq}
the sequence $(b_n)_n$ converges computably to $0$.
Let $g:\IN\to\IN$ be a computable modulus of convergence of this sequence.
Remember that this implies that for $m,n\in\IN$ with $m\geq g(n)$ we
have $b_m\leq 2^{-n}$, hence, $f(m) \geq n$.
Hence, a number $n$ is an element of $A$ if, and only if, $n$ is an element of 
$\{f(0),\ldots,f(g(n+1)-1)\}$.
This shows that $A$ is decidable.
\end{proof}

Remember that a real number $\alpha$ is called {\em computably approximable} if
there exists a computable sequence $(q_n)_n$ of rational numbers that converges to $\alpha$,
and that by $\CA$ we denote the set of all computably approximable numbers.

\begin{corollary}
\begin{enumerate}
\item
$\LC \not\subseteq \NC$ and $\NC\not\subseteq \LC$.
\item
$\NC\subsetneq \CA$.
\end{enumerate}
\end{corollary}

\begin{proof}
Let $K\subseteq \IN$ be a computably enumerable set that is not decidable.
Then the number $0.K$ is strongly left-computable, hence also left-computable, but not computable.
By Proposition~\ref{proposition:strongly-lc} it is not nearly computable.
This shows $\LC \not\subseteq \NC$.

The inclusion $\NC\subseteq \CA$ is clear.
As any left-computable number is computably approximable, 
the number $0.K$ is an element of $\CA$ but not of $\NC$.
Hence, $\NC\subsetneq \CA$.

By Theorem~\ref{theorem:downey-laforte} there exists a real number $\alpha$
that is left-computable and nearly computable but not computable.
Then the number $-\alpha$ is not left-computable (otherwise $\alpha$ would even be computable)
but  nearly computable (this follows from Proposition~\ref{prop:ncc-seq-vectorspace}).
This shows $\NC\not\subseteq \LC$.
\end{proof}

Let us deduce some observations concerning the possible positions of the elements of the following set
\[ \Theta:=\left\{A \subseteq \IN ~:~ 0.A \text{ is a nearly computable number}\right\} \]
in the arithmetical hierarchy.
Let $\Sigma_n^0$, $\Pi_n^0$, and $\Delta_n^0$ be the usual classes of the arithmetical
hierarchy; see, e.g., \cite{Soa87}.

\begin{corollary}
\begin{enumerate}
\item
$\Sigma_0^0 = \Pi_0^0 = \Delta_0^0 =\{A \subseteq \IN ~:~ A \text{ is decidable}\} \subsetneq \Theta$.
\item
$\Sigma_1^0 \cap \Theta = \Pi_1^0 \cap \Theta = \Sigma_0^0$.
\item
$\Theta \subsetneq  \Delta_2^0$.
\end{enumerate}
\end{corollary}

\begin{proof}
\begin{enumerate}
\item
On the one hand, remember that for $A\in\IN$ the number $0.A$ is computable iff $A$ is decidable
and that any computable number is nearly computable.
On the other hand, by Theorem~\ref{theorem:downey-laforte} there exists a real number that is 
nearly computable but not computable.
\item
The equality $\Sigma_1^0 \cap \Theta = \Sigma_0^0$
follows directly from Proposition~\ref{proposition:strongly-lc}.
The inclusion $\Sigma_0^0 \subseteq \Pi_1^0 \cap \Theta$ follows from the first assertion.
For the inclusion $\Pi_1^0 \cap \Theta \subseteq \Sigma_0^0$ let us consider
a set $A\in \Pi_1^0 \cap \Theta$. Then the set $B:=\IN\setminus A$
is an element of $\Sigma_1^0$, and the number $0.B=1-0.A$
is nearly computable. We conclude $B\in\Sigma_1^0\cap \Theta=\Sigma_0^0$, hence,
$A\in\Sigma_0^0$.
\item
This follows from $\NC\subsetneq \CA$ and from the fact that
a set $A\subseteq \IN$ is an element of $\Delta_2^0$ if, and only if,
the number $0.A$ is computably approximable; see Ho~\cite{Ho99}.
\qedhere
\end{enumerate}
\end{proof}

The second observation of the kind described above
(that asking a real number to be left-computable and nearly computable and
to satisfy some other effectivity condition forces the number to be computable)
that we wish to mention is the following result by Wu~\cite{Wu01}.

\begin{proposition}[{Wu~\cite{Wu01}}]
There exists an undecidable, c.e. set $B \subseteq \IN$ such that
every left-computable and nearly computable number $\alpha$ with 
$\alpha \leqT B$ is computable.
\end{proposition}

Here $\leqT$ denotes Turing reduction; see Soare~\cite{Soa87}.
A set $A\subseteq \IN$ is {\em Turing reducible} to a set $B\subseteq \IN$, written $A \leqT B$,
if there exists an oracle Turing machine which with oracle $B$ computes $A$.
By $\alpha \leqT B$ for a real number $\alpha$ we mean $A \leqT B$ where $A\subseteq\IN$
is the unique coinfinite set with $(\alpha - 0.A)\in\IZ$.
The proof of Wu's result is a complicated $\emptyset'''$-priority argument.

\section{Nearly Computable Numbers, a Genericity Notion and Randomness Notions}
\label{section:nearly-computable-randomness}

The third observation of the kind described in the previous section
(that asking a real number to be left-computable and nearly computable and
to satisfy some other effectivity condition forces the number to be computable)
that we wish to mention is the following result by Hoyrup~\cite{Hoy17}.

\begin{proposition}[{Hoyrup~\cite[Prop.~4.3.2]{Hoy17}}]
\label{prop:Hoy17}
If a real number is left-computable and nearly computable 
then it is either computable or generic from the right.
\end{proposition}

\begin{corollary}[{Hoyrup~\cite{Hoy17}}]
\label{cor:Hoy17}
If a real number is left-computable and nearly computable 
then it is either computable or weakly $1$-generic.
\end{corollary}

Note that the assertion in the corollary is equivalent to the following: 
If a real number is left-computable, nearly computable, and not weakly $1$-generic
then it is computable.

Let us explain the notions appearing in these assertions.
A subset $U\subseteq\IR$ of the real line is called {\em effectively open}
if there exists a c.e.~set $A\subseteq\IN$ such that
$U=\bigcup_{n\in A} B^{(1)}_n$.
Hoyrup~\cite{Hoy17} calls a real number $x$ {\em generic from the right} 
if for every effectively open set $U\subseteq \IR$ either $x\in U$
or there exists a number $\varepsilon>0$ with $[x,x+\varepsilon) \cap U = \emptyset$.
A real number $x$ is called {\em weakly $1$-generic} 
if it is an element of every dense effectively open subset $U\subseteq \IR$~\cite{Kur83}.
Clearly, every real number that is generic from the right is
weakly $1$-generic as well~\cite[Prop.~4.3.1]{Hoy17}. 
Hence, Corollary~\ref{cor:Hoy17} follows directly from Proposition~\ref{prop:Hoy17}.
Note that a computable number $x$ cannot be weakly $1$-generic
because for a computable number $x$ the set $\IR\setminus\{x\}$
is a dense and effectively open subset of $\IR$.
One cannot delete the assumption that the real number should
be left-computable from Proposition~\ref{prop:Hoy17}
because by Theorem~\ref{theorem:downey-laforte} and Corollary~\ref{cor:NC-real-closed-field}
there exist right-computable numbers that are nearly
computable but not computable (a real number $x$ is called {\em right-computable}
if $-x$ is left-computable), and a right-computable number cannot be
generic from the right (because for a right-computable number $x$ the set $(x,\infty)$ is an effectively open set).
But one can delete this assumption from Corollary~\ref{cor:Hoy17}.

\begin{theorem}
\label{theorem:nearly-computable-weakly-1-generic}
If a real number is nearly computable 
then it is either computable or weakly $1$-generic.
\end{theorem}

\begin{proof}
For the sake of a contradiction,
let us assume that there exists a real number $y$ that is nearly computable
but neither computable nor weakly $1$-generic.
Let $U\subseteq \IR$ be a dense, effectively open set with $y\not\in U$.
There exists a total computable function $f:\IN\to\IN$ such that $U=\bigcup_{n\in\IN} B^{(1)}_{f(n)}$.
Let $(x_n)_n$ be a computable sequence of rational numbers that converges to $y$
and is nearly computably Cauchy.

We define a strictly increasing computable function $s:\IN\to\IN$
as follows. We set $s(0):=0$. Let us consider a number $n>0$
and let us assume that we have fixed the values $s(0)< \ldots < s(n-1)$.
Now we search for a number $m>s(n-1)$ such that
\[  x_m \not\in \bigcup_{i=0}^{n-1} \overline{B}_{f(i)} , \]
where $\overline{B}_n$ is the closure of $B^{(1)}_n$, that is,
$\overline{B}_n=[\nu_\IQ \pi_1(n) - 2^{-\pi_2(n)},\nu_\IQ \pi_1(n) - 2^{-\pi_2(n)}]$,
for all $n\in\IN$.
Such a number $m$ exists because $y\not\in U$ and $y$ is not rational, hence,
$y\not\in \bigcup_{i=0}^{n-1} \overline{B}_{f(i)}$,
and because the sequence $(x_k)_k$ converges to $y$.
Then we set $s(n):=m$.

The sequence $(z_n)_n$ defined by $z_n:=x_{s(n)}$ 
is a computable sequence of rational numbers 
that converges to $y$ and is nearly computably Cauchy
by Lemma~\ref{lemma:nearly-computably-convergent-1}.
And it has the property that
\[ z_n \not\in \bigcup_{i=0}^{n-1} \overline{B}_{f(i)} , \]
for all $n\in\IN$.
Let $g:\IN\to\IN$ be a computable modulus of convergence of the sequence $(|z_{n+1}-z_n|)_n$.
Then, for all $m,n$ with $m\geq g(n)$, we have $|z_{m+1}-z_m| \leq 2^{-n}$.
We define a computable function $\mathit{label}:\IN\to\{\mathrm{left},\mathrm{right}\}$ by
$r(n):=\max\{n+1,g(\pi_2(f(n)))\}$ and
\[ \mathit{label}(n):=\begin{cases}
     \mathrm{left} & \text{if } \nu_\IQ(\pi_1(f(n))) < z_{r(n)}, \\
     \mathrm{right} & \text{if } z_{r(n)} < \nu_\IQ(\pi_1(f(n))), 
    \end{cases}
\]
for $n\in\IN$ 
(note that $\nu_\IQ(\pi_1(f(n)))$ is the midpoint of the closed interval $\overline{B}_{f(n)}$
and $r(n)>n$, hence, $z_{r(n)} \not\in \overline{B}_{f(n)}$, for all $n$).
Now note that, for all $m\geq r(n)$,
on the one hand $z_m\not\in \overline{B}_{f(n)}$,
on the other hand, 
$|z_{m+1} - z_m|\leq 2^{-\pi_2(f(n))}$ and the length of the interval $\overline{B}_{f(n)}$
is equal to $2\cdot 2^{-\pi_2(f(n))}$.
Hence, for all $n$ with $\mathit{label}(n)=\mathrm{left}$ 
and for all $m\geq r(n)$ the interval $\overline{B}_{f(n)}$ is to the left of $z_m$.
We conclude that for all $n$ with $\mathit{label}(n)=\mathrm{left}$
this interval is to the left of $y$.
Similarly, for all $n$ with $\mathit{label}(n)=\mathrm{right}$
the interval $\overline{B}_{f(n)}$ is to the right of $y$.
As the set $U$ is a dense subset of $\IR$, we can use the function $\mathit{label}$ in order
to compute $y$ with arbitrary precision, hence, $y$ is computable. Contradiction.
\end{proof}

For the notions ``hyperimmune'' and ``Martin-L{\"o}f random'',
that appear in the following corollary, the reader is referred to~\cite{Nie09,DH10}.

\begin{corollary}
\label{theorem:hyperimmune-stronger}
\begin{enumerate}
\item
If a real number is nearly computable then it is either computable or hyperimmune.
\item
If a real number is nearly computable then it is not Martin-L{\"o}f random.
\end{enumerate}
\end{corollary}

Note that the first assertion is equivalent to the following: 
If a real number is nearly computable and not hyperimmune
then it is computable.

\begin{proof}
The first assertion follows from Theorem~\ref{theorem:nearly-computable-weakly-1-generic}
and from the following result by Kurtz~\cite[Corollary 2.7]{Kur83}:
every weakly $1$-generic real is hyperimmune
(for a short proof of this fact see \cite[Prop.~1.8.48]{Nie09}).
The second assertion follows from the first because 
a Martin-L{\"o}f random real is neither computable nor hyperimmune; see, e.g., \cite{SW05}.
\end{proof}

This corollary strengthens Theorem 6 by Stephan and Wu~\cite[Theorem 6]{SW05},
who had shown the same assertions under the additional assumption that the
real number is left-computable.
This was the fourth observation of the kind described in the previous section
(that asking a real number to be left-computable and nearly computable and
to satisfy some other effectivity condition forces the number to be computable)
that we wanted to mention.

In order to state the next results we need to introduce some further notions.
Let $\Sigma:=\{0,1\}$ be the binary alphabet.
For any function $f:\subseteq\Sigma^*\to\Sigma^*$ defined on a subset of $\Sigma^*$
and with values in $\Sigma^*$
the {\em Kolmogorov complexity $K_f(\tau)$} of a binary string $\tau$ is
\[ K_f(\tau):=\begin{cases}
   \infty & \text{if there is no } \sigma \in \dom(f) \text{ with } f(\sigma)=\tau, \\
   \min\{|\sigma| ~:~ \sigma\in\dom(f), f(\sigma)=\tau\} & \text{otherwise}.
   \end{cases}
\]
A computable function $f:\subseteq\Sigma^*\to\Sigma^*$ is called {\em universal}
if its domain $\dom(f)$
is a prefix-free subset of $\Sigma^*$ and if for every computable function $g:\subseteq\Sigma^*\to\Sigma^*$
with prefix-free domain there exists a constant $c\in\IN$ such that
for any binary string $\tau$ with $K_g(\tau)<\infty$ we have $K_f(\tau)\leq K_g(\tau)+c$.
It is well known that there exists a universal function $f:\subseteq\Sigma^*\to\Sigma^*$;
see, e.g., \cite{DH10,Nie09}.
Note that if $f:\subseteq\Sigma^*\to\Sigma^*$ is a universal function
then $K_f(\tau)<\infty$, for all $\tau\in\Sigma^*$.
In the rest of this section we fix a universal function $U:\subseteq\Sigma^*\to\Sigma^*$
and write $K$ instead of $K_U$.

For any finite binary string $w\in\Sigma^*$ let $0.w$ be the rational number in $[0,1]$
with the finite binary expansion $w$, that is,
$0.w:=\sum_{n=0}^{|w|-1} 2^{-n-1} \cdot w(n)$. Similarly, for any $\alpha\in\Sigma^\IN$
let $0.\alpha$ be the real number in $[0,1]$ defined by $0.\alpha:=\sum_{n=0}^\infty 2^{-n-1}\cdot \alpha(n)$.
In this case $\alpha$ is called a {\em binary representation} of the real number $0.\alpha$.
Note that only the real numbers in $[0,1]$ of the form $z/2^k$ for $k\in\IN$ and $z\in\IZ$
with $0\leq z \leq 2^k$ have two binary representations.
All other real numbers in $[0,1]$ have exactly one binary representation.

\begin{definition}
\begin{enumerate}
\item
An infinite binary sequence $\alpha$ is called {\em K-trivial}
if there exists a constant $c\in\IN$ such that, for all $n\in\IN$,
$K(\alpha\upharpoonright n) \leq K(0^n) + c$, where $0^n$
is the string of length $n$ consisting only of zeros.
\item
An infinite binary sequence $\alpha$ is called {\em infinitely often K-trivial}
if there exists a constant $c\in\IN$ such that, for infinitely many $n\in\IN$,
$K(\alpha\upharpoonright n) \leq K(0^n) + c$.
\item
Following \cite{SW05} we call an infinite binary sequence $\alpha$ {\em strongly Kurtz random}
if there does not exist a total computable function $f:\IN\to\IN$ such that, for all $n\in\IN$,
$K(\alpha\upharpoonright f(n)) < f(n)-n$.
\end{enumerate}
\end{definition}

A real number $x \in[0,1]$ is called {\em K-trivial} respectively {\em infinitely often K-trivial}
respectively {\em strongly Kurtz random} if
at least one (and then all) of its binary representations is (are) K-trivial respectively strongly Kurtz random.

\begin{corollary}
Every nearly computable number in $[0,1]$ is infinitely often K-trivial.
\end{corollary}

\begin{proof}
This follows from Theorem~\ref{theorem:nearly-computable-weakly-1-generic},
from the fact that every computable real number in $[0,1]$ is K-trivial,
hence, infinitely often K-trivial,
and from the following result by Barmpalias and Vlek~\cite[Theorem 2.9]{BV11}:
every weakly $1$-generic real in $[0,1]$ is infinitely often K-trivial.
\end{proof}

This raises the question whether a nearly computable number can or must be K-trivial.
For nearly computable numbers that are left-computable this question
was answered by Stephan and Wu~\cite{SW05}.
The following result by them is the fifth observation of the kind described in the previous section
(that asking a real number to be left-computable and nearly computable and
to satisfy some other effectivity condition forces the number to be computable)
that we wish to mention.

\begin{proposition}[{Stephan and Wu~\cite[Theorems 3 and 4]{SW05}}]
\begin{enumerate}
\item
If a real number in $[0,1]$ is left-computable and nearly computable 
then it is either computable or strongly Kurtz random.
\item
If a real number in $[0,1]$ is left-computable, nearly computable, and K-trivial then it is computable.
\end{enumerate}
\end{proposition}

Note that the first assertion is equivalent to the following: 
If a real number is left-computable, nearly computable, and not strongly Kurtz random
then it is computable.

We are going to show that, similarly as above, this proposition can be strengthened as follows:
One can omit the assumption that the real number is left-computable.

\begin{theorem}
\label{theorem:StephanWu-stronger}
\begin{enumerate}
\item
If a real number in $[0,1]$ is nearly computable 
then it is either computable or strongly Kurtz random.
\item
If a real number in $[0,1]$ is nearly computable and K-trivial then it is computable.
\end{enumerate}
\end{theorem}

\begin{proof}%[Proof of Theorem~\ref{theorem:StephanWu-stronger}]
The second assertion follows from the first assertion because a K-trivial binary sequence
cannot be strongly Kurtz random.
For completeness sake we prove this, similarly to the argument in \cite[Page 468]{SW05}.
Let $\alpha \in\Sigma^\IN$ be $K$-trivial.
Then there exists a constant $c\in\IN$ such that, for all $n\in\IN$,
$K(\alpha\upharpoonright n) \leq K(0^n) + c$.
And it is well known that there exists a constant $d\in\IN$ such that,
for all $n\in\IN$, $K(0^n) < \frac{n}{2} + d$.
The function $f:\IN\to\IN$ defined by $f(n) := 2\cdot (n+c+d)$
is computable and satisfies, for all $n\in\IN$,
\[ K(\alpha\upharpoonright f(n)) \leq K(0^{f(n)}) + c < \frac{f(n)}{2} + d + c = f(n)-n . \]
Hence, $\alpha$ is not strongly Kurtz random.

We prove the first assertion.
Let $\alpha\in\Sigma^\IN$ be an infinite binary sequence that is not strongly Kurtz random and such that the
real number $0.\alpha$ is nearly computable. We wish to show that the real number $0.\alpha$ is computable.
This is clear if $\alpha$ contains only finitely many zeros or only finitely many ones.
So, in the following we assume that $\alpha$ contains infinitely many zeros and infinitely many ones.
Let $f:\IN\to\IN$ be a total computable function such that $K(\alpha\upharpoonright f(n))<f(n)-n$, for all $n$.
This implies $f(n)>n$, for all $n$.
Let $(a_n)_n$ be a computable sequence of rational numbers in $[0,1]$ that is nearly computably Cauchy
and converges to $0.\alpha$.
For each $n$ we can compute a finite binary string $u_n$ with $|0.u_n - a_n| < 2^{-n}$.
Then the sequence $(u_n)_n$ of binary strings is computable, and the sequence
$(0.u_n)_n$ of rational numbers is nearly computably Cauchy as well (by Lemma~\ref{lemma:ncc-pair-close})
and converges to $0.\alpha$ as well.
For each $n\in\IN$ let $\beta_n\in\Sigma^\IN$ be the infinite binary sequence obtained
by appending to $u_n$ infinitely many zeros, that is, $\beta_n:=u_n0^\IN$.
As $\alpha$ contains infinitely many zeros and infinitely many ones,
the sequence $(\beta_n)_n$ converges to $\alpha$
(in the sense that for every $\ell\in\IN$ there exists an $m\in\IN$ such that, for all $n\geq m$,
$\beta_n \upharpoonright \ell = \alpha \upharpoonright \ell$).
We define a strictly increasing and computable function $s:\IN\to\IN$ as follows.
In order to define $s(0)$ we search for a number $t$ such that $K(\beta_t\upharpoonright f(0))<f(0)$
(the search is done using a fixed computable enumeration
of the domain $\dom(U)$ of the universal function $U$).
Once we have found such a $t$ we set $s(0):=t$. Note that
$K(\beta_{s(0)}\upharpoonright f(0))<f(0)-0$.
Let us assume that we have defined $s(0),\ldots,s(n)$ so that
\begin{enumerate}
\item
$s(0)< s(1) < \ldots < s(n)$,
\item
for all $k,m$ with $0 \leq k \leq m \leq n$, $K(\beta_{s(m)}\upharpoonright f(k))<f(k)-k$.
\end{enumerate}
In order to define $s(n+1)$ we search for a number $t>s(n)$ such that, for all $k\leq n+1$,
$K(\beta_t\upharpoonright f(k))<f(k)-k$.
Once we have found such a $t$ we set $s(n+1):=t$.
As, by assumption $K(\alpha\upharpoonright f(n))<f(n)-n$, for all $n$,
in this way a strictly increasing, computable function $s:\IN\to\IN$ is defined such that,
\[ \text{for all } m,n, \text{ if } m\geq n \text{ then } K(\beta_{s(m)}\upharpoonright f(n))<f(n)-n. \]

The sequence $(0.u_{s(n)})_n$ is a computable sequence of rational numbers.
As the sequence $(0.u_n)_n$ is nearly computably Cauchy, the sequence
$(0.u_{s(n+1)} - 0.u_{s(n)})_n$ converges computably to $0$. 
Let $g:\IN\to\IN$ be a computable
and strictly increasing 
modulus of convergence of this sequence.
We claim that, for all $n$, 
\begin{equation}
\label{eq:StephanWu-claim}
 |0.u_{s(g(f(n)))} - 0.\alpha| \leq 2^{-n} .
\end{equation}
As the sequence of rational numbers $(0.u_{s(g(f(n)))})_n$ is computable,
this claim shows that $0.\alpha$ is computable.
Let us prove the claim \eqref{eq:StephanWu-claim}.
Let us consider some $n\in\IN$.
Let 
\[ U_n :=\{w \in \Sigma^{f(n)} ~:~ \min\{0.\alpha, 0.u_{s(g(f(n)))}\}
  < 0.w < \max\{0.\alpha, 0.u_{s(g(f(n)))}\} \} . \]
Let 
\[ V_n := \{w \in \Sigma^{f(n)} ~:~ (\exists m \geq g(f(n))) \ w = \beta_{s(m)}\upharpoonright f(n) \} . \]
We make the following three observations.
\begin{enumerate}
\item
$|0.u_{s(g(f(n)))} - 0.\alpha| \leq (|U_n| + 1) \cdot 2^{-f(n)}$
(here $|U_n|$ is the number of elements of the finite set $U_n$).
This is an immediate consequence of the definition of $U_n$.
\item
$U_n \subseteq V_n$. This is due to the fact that, for all $m\geq g(f(n))$,
$|0.u_{s(m+1)} - 0.u_{s(m)}| \leq 2^{-f(n)}$.
The rational numbers $0.u_{s(m)}$ for $m\geq g(f(n))$ converge to $\alpha$,
and they do this in small steps so that
all intervals of the form $[0.w,0.w+2^{-f(n)})$ for $w\in U_n$ must be visited by them.
\item
$|V_n| \leq 2^{f(n)-n-1}$. 
Indeed, on the one hand, for all $m\geq n$,
$K(\beta_{s(m)} \upharpoonright f(n))<f(n)-n$,
and, if $m\geq g(f(n))$, then, due to $f(n)>n$ and the assumption that $g$ is strictly increasing,
$m\geq g(f(n)) > g(n)\geq n$.
On the other hand, it is clear that for any $k>0$ there are at most
$2^{k-1}$ strings $w$ with $K(w) < k$.
\end{enumerate}
Putting these observations together we conclude
\[  |0.u_{s(g(f(n)))} - 0.\alpha|
   \leq (|U_n| + 1) \cdot 2^{-f(n)}
   \leq (|V_n| + 1) \cdot 2^{-f(n)}
   \leq (2^{f(n)-n-1} + 1) \cdot 2^{-f(n)}
   \leq 2^{-n}.
\]
This ends the proof of the first assertion and, thus, the proof of the theorem.
\end{proof}

\section{Promptly Simple Sets and Nearly Computable Numbers}
\label{section:promptly-simple}

In this section we are going to strengthen the following result by Downey and LaForte.

\begin{proposition}[{Downey and LaForte~\cite[Theorem 11]{DL02}}]
\label{proposition:DLF-promptly-simple}
If a real number is left-computable and nearly computable 
then it does not have promptly simple Turing degree.
\end{proposition}

A set $A\subseteq\IN$ is called {\em simple} if it is c.e., coinfinite,
and its complement does not contain an infinite c.e.~subset.
This notion was strengthened as follows by Maass~\cite{Maa82}.

\begin{definition}
A set $A\subseteq \IN$ is called {\em promptly simple} if it is c.e., coinfinite,
and if there is a computable function $h:\IN\to\IN$ such that,
for every $e\in\IN$, if $W_e$ is infinite then there exist $s>0$ and $x\in\IN$
such that $x \in (W_e[s] \setminus W_e[s-1]) \cap A[h(s)]$.
\end{definition}

Clearly, every promptly simple set is simple.
A Turing degree is called {\em computably enumerable} or {\em c.e.} if it contains a c.e. set.
A Turing degree is called {\em promptly simple} if it contains a promptly simple set.
Clearly, every promptly simple Turing degree is a c.e. Turing degree.

For a real number $\beta$, by the {\em Turing degree of $\beta$} we mean the Turing degree
of the unique coinfinite set $B\subseteq\IN$ with $(\beta - 0.B)\in\IZ$.
Note that this is identical with the Turing degree of the set $\{n\in\IN ~:~ \nu_\IQ(n)<\beta\}$.
Hence, the Turing degree of any left-computable number is a c.e. Turing degree.

As within the c.e.~Turing degrees
the set of all promptly simple Turing degrees is upwards closed
with respect to Turing reduction~\cite[Corollary XIII.1.9]{Soa87}
(that means, if $\mathbf{a}$ is a promptly simple Turing degree
and $\mathbf{b}$ is a c.e.~Turing degree with $\mathbf{a} \leqT \mathbf{b}$
then $\mathbf{b}$ is a prompty simple Turing degree as well),
Proposition~\ref{proposition:DLF-promptly-simple} by Downey and LaForte~\cite[Theorem 11]{DL02}
is equivalent to the following assertion:

\begin{quote}
{\em If $\beta$ is a left-computable and nearly computable number then there does not exist a 
promptly simple set $A$ with $A\leqT\beta$.}
\end{quote}

We are going to show that this assertion can be strengthened as follows: One can omit
the assumption that the real number is left-computable.

\begin{theorem}
\label{theorem:promptlysimple}
If $\beta$ is a nearly computable number then there does not exist a 
promptly simple set $A$ with $A\leqT\beta$.
\end{theorem}

For the proof the following definition and
the following observation due to Ambos-Spies et al.~\cite{AJSS84} will be useful.

\begin{definition}
\label{definition:strong-array}
A sequence $B:\IN\to\mathcal{P}(\IN)$ of finite subsets $B_n\subseteq\IN$ is called a {\em strong array} if
the function $b:\IN\to\IN$ defined by $b(n):=\sum_{i\in B_n} 2^i$ is computable.
\end{definition}

\begin{lemma}[Slowdown Lemma, Ambos-Spies et al.~\cite{AJSS84}]
Let $(U_k)_k$ be a strong array such that, for all $e,n$,
$U_{\langle e,n\rangle}\subseteq U_{\langle e,n+1 \rangle}$.
For each $e$, let $U_e:=\bigcup_{n\in\IN} U_{\langle e,n \rangle}$.
Then there exists a total computable function $g:\IN\to\IN$ such that, for all $e$,
$W_{g(e)}=U_e$ and, for all $e\in\IN$ and $s>0$,
$(U_{\langle e,s \rangle} \setminus U_{\langle e,s-1 \rangle}) \cap W_{g(e)}[s] = \emptyset$.
\end{lemma}

In the rest of this section we identify a subset $A\subseteq \IN$
with its characteristic function $\chi_A:\IN\to\{0,1\}$.
So, for a number $n\in\IN$ we write
\[ A(n) := \chi_A(n) = \begin{cases}
     0 & \text{if } n \not\in A, \\
     1 & \text{if } n\in A,
     \end{cases} \]
and $A\upharpoonright n$ is the binary string $A(0)\ldots A(n-1)$ of length $n$.
In the proof of Theorem~\ref{theorem:promptlysimple} we shall use the following characterization of
Turing reducibility. 

\begin{lemma}
\label{lemma:Turing-reduction}
For two sets $A,B\subseteq \IN$ the following two conditions are equivalent.
\begin{enumerate}
\item
$A \leqT B$.
\item
There exists a c.e.~set $C \subseteq \{0,1\}^* \times\IN\times\{0,1\}$ with the following properties.
\begin{enumerate}
\item
For any triples $(w,n,z), (w',n',z') \in C$, if $w$ is a prefix of $w'$ and $n=n'$ then $z=z'$.
\item
For any $n$ there exists a number $\ell$ with $(B\upharpoonright \ell,n,A(n))\in C$.
\end{enumerate}
\end{enumerate}
\end{lemma}

\begin{proof}[Proof of Theorem~\ref{theorem:promptlysimple}]
Let us consider a computably approximable number $\beta$
and a promptly simple set $A$ with $A \leqT \beta$. 
We can assume without loss of generality $0\leq \beta < 1$.
Let $B\subseteq \IN$ be the uniquely determined set with $0.B=\beta$
(note that $A$ is simple, hence, undecidable;
therefore, $A \leqT \beta$ implies that $\beta$ is not computable;
therefore there exists exactly one set $B\subseteq \IN$ with $\beta=0.B$,
and this set is undecidable, hence, coinfinite).
It is sufficient to show that $\beta$ is not nearly computable.
Let us consider an arbitrary computable sequence $(b_n)_n$ of rational numbers that converges to $\beta$.
It is sufficient to show that this sequence is not nearly computably Cauchy.
Hence, it is sufficient to show that there exists a strictly increasing computable function $r:\IN\to\IN$ such that
the sequence $(b_{r(n+1)} - b_{r(n)})_n$ (which converges to $0$, of course)
does not converge computably.
Thus, it is sufficient to construct a computable strictly increasing function $r:\IN\to\IN$ such that for all $e\in\IN$ 
the function $\varphi_e$ is not a modulus of convergence of the sequence $(b_{r(n+1)} - b_{r(n)})_n$.
So, we wish to construct a computable strictly increasing function $r:\IN\to\IN$ such that for all $e\in\IN$ 
the following requirement $R_e$ is satisfied:
\[ R_e: \text{if } \varphi_e \text{ is total then }
   (\exists m,n\in\IN) \ (m\geq \varphi_e(n) \text{ and } |b_{r(m+1)} - b_{r(m)}| > 2^{-n}) . \]

Let $h:\IN\to\IN$ be a computable function that shows the prompt simplicity of $A$, that is, a function such that
for every $e\in\IN$, if $W_e$ is infinite then there exist $s>0$ and $x\in\IN$
such that $x \in (W_e[s] \setminus W_e[s-1]) \cap A[h(s)]$.
Let $C \subseteq \{0,1\}^* \times\IN\times\{0,1\}$ be a c.e. set as in Lemma~\ref{lemma:Turing-reduction}
applied to $A$ and to the coinfinite set $B$ with $0.B=\beta$, that is,
$C$ shows as in Lemma~\ref{lemma:Turing-reduction} that $A\leqT \beta$.

We shall construct the desired function $r$ in a sequence of stages; the value $r(s)$ will be defined in stage $s$.
At the same time we are going to construct a certain strong array $(U_k)_k$ 
such that for all $e,n$, $U_{\langle e,n \rangle}\subseteq U_{\langle e,n+1 \rangle}$.
We shall write $U_e[s]$ for $U_{\langle e,s \rangle}$.
Let $g:\IN\to\IN$ be a function as in the Slowdown Lemma, applied to $(U_k)_k$.

We start in stage $s=0$ by setting $r(0):=0$ and $U_e[0]:=\emptyset$, for all $e\in\IN$.
Then we continue with stage $s=1$.

Let us now describe what we do in a stage $s>0$.
First we check whether there exists a number $e<s$ with the following properties:
\begin{enumerate}
\item
$\{0,\ldots, |U_e[s-1]|\} \subseteq W_e[s]$
(note that the requirement $R_e$ is nontrivial only for numbers $e$ with $W_e=\IN$),
\item
$(\forall n \in W_e[s])\, (\forall m < s-1) \,
  (\text{if } m \geq \varphi_e(n) \text{ then } |b_{r(m+1)} - b_{r(m)}| \leq 2^{-n})$
(this means that the restriction of $\varphi_e$ to the set $W_e[s]$ looks as if $\varphi_e$
might be a modulus of convergence of the sequence $(b_{r(m+1)} - b_{r(m)})_m$),
\item
there exists a binary string $w$ satisfying $|w|\leq s$ and $0.w < b_{r(s-1)} < 0.w+2^{-|w|}$ 
(this implies that $0.w$ is a prefix of the binary representation of $b_{r(s-1)}$)
and such that 
the smallest number $y$ with $2^{-y} < \min\{b_{r(s-1)} - 0.w, 0.w+2^{-|w|} - b_{r(s-1)}\}$
is an element of $W_e[s]$ and satisfies $\varphi_e(y)<s$
and such that there exists a number $x\in\{0,1,\ldots,s\} \setminus U_{e}[s-1]$
with $A(x)[s]=0$ and $(w,x,0)\in C[s]$.
\end{enumerate}
If there does not exist such a number $e$ then we set $r(s):=r(s-1)+1$ and
$U_{\widetilde{e}}[s]:=U_{\widetilde{e}}[s-1]$, for all $\widetilde{e}$,
we end this stage and continue with stage $s+1$.

If there exists such a number $e$ then we fix the smallest such $e$,
say that $R_e$ {\em receives attention at stage $s$} and proceed as follows.
First we fix the shortest binary string $w$ satisfying the third condition above.
Then we fix the smallest number $x$ as in the third condition above.
We enumerate $x$ into the set $U_e$, that is, we set $U_e[s] := U_e[s-1] \cup\{x\}$
(and we set $U_{\widetilde{e}}[s]:=U_{\widetilde{e}}[s-1]$ for all $\widetilde{e} \in \IN\setminus\{e\}$).
We search for the smallest number $t$ with $x\in W_{g(e)}[t]$
(note that the Slowdown Lemma, which ensures the existence of the function $g$,
is proved by the Recursion Theorem;
for a justification why we can use the function $g$ here,
in the definition of $(U_k)_k$,
see~\cite[Chapter II.3]{Soa87}).
Such a number $t$ exists, and it is strictly greater than $s$.
Now we check whether $x\in A[h(t)]$. If this is not the case then
we say that the current attempt to satisfy the requirement $R_e$ has failed,
we set $r(s):=r(s-1)+1$, and we continue with stage $s+1$. 

If, on the other hand, $x\in A[h(t)]$ then let $u$ be the smallest number strictly greater than $r(s-1)$
such that there exists a binary string $\widetilde{w}$
with $0.\widetilde{w} < b_u < 0.\widetilde{w} + 2^{-|\widetilde{w}|}$ and
$(\widetilde{w},x,1) \in C[u]$. We set $r(s):=u$ and continue with stage $s+1$.
In this case we say that the current attempt to satisfy the requirement $R_e$ was successful.
This ends the description of the algorithm.

We come to the verification. 
It is clear that the constructed function $r:\IN\to\IN$ is strictly increasing and computable. 
We have to show that it satisfies all requirements $R_e$.
For the sake of a contradiction, let us assume that this is not the case.
Let $e$ be the smallest number such that the requirement $R_e$ is not satisfied.
Then $\varphi_e$ is a total function and a modulus of convergence of the sequence
$(b_{r(m+1)} - b_{r(m)})_m$.

First, we show that the requirements $R_d$ with $d<e$ will receive attention only finitely often.
For $d$ such that $\varphi_d$ is not a total function this is due to the condition
$\{0,\ldots, |U_d[s-1]|\} \subseteq W_d[s]$
(if $R_d$ would receive attention infinitely often then $U_d$ would be infinite
and this condition would imply $W_d=\IN$).
And if $d<e$ has the property that $\varphi_d$ is total and $R_d$ is satisfied then 
there exist $m,n$ with $m\geq \varphi_d(n)$ and $|b_{r(m+1)} - b_{r(m)}| > 2^{-n}$,
and for sufficiently large $s$ we have $n \in W_d[s]$ and $m<s-1$.
So, we have shown that the requirements $R_d$ with $d<e$ receive attention only finitely often.

Next, we show that $R_e$ will receive attention infinitely often.
For the sake of a contradiction let us assume
that this is not the case and that there is some $s_0>0$
such that no $R_d$ with $d\leq e$ will receive attention at any stage $s\geq s_0$.
Then for $s\geq s_0$ the set $U_e[s-1]$ does not change anymore,
hence, $U_e[s-1]=U_e$, hence, for sufficiently
large $s$ we have $\{0,\ldots, |U_e[s-1]|\} \subseteq W_e[s]$.
As we assume that $R_e$ is not satisfied,
the function $\varphi_e$ is a total function and a modulus of convergence
of the sequence $(b_{r(m+1)} - b_{r(m)})_m$.
Hence, the second condition above is satisfied as well.
As $A$ is coinfinite let us consider some $x \in\IN\setminus A$ that is not an element of $U_e$.
Then $A(x)=0$, hence, for all $s$, $A(x)[s]=0$, and for sufficiently large $s$ we have $x \leq s$, hence,
$x\in\{0,1,\ldots,s\}\setminus U_e[s-1]$.
There exists a number $\ell$ such that  $(B\upharpoonright \ell,x,0)\in C$.
We set $w:=B\upharpoonright \ell$.
For sufficiently large $s$ we have $|w|\leq s$ and $(w,x,0) \in C[s]$.
As the sequence $(b_k)_k$ converges to $\beta=0.B$ for $k$ tending to $\infty$ and $B$ is not decidable,
for sufficiently large $s$ we have $0.w< b_{r(s-1)} < 0.w+2^{-|w|}$.
Furthermore, for sufficiently large $s$ the smallest number $y$
with $2^{-y} < \min\{b_{r(s-1)} - 0.w, 0.w+2^{-|w|} - b_{r(s-1)}\}$
is an element of $W_e[s]$ (because $W_e=\IN$) and satisfies $\varphi_e(y)< s$.
We conclude that $R_e$ will receive attention at a sufficiently large stage $s\geq s_0$. Contradiction.
We have shown that $R_e$ will receive attention infinitely often.

This implies that the set $U_e=W_{g(e)}$ is infinite. 
Now the fact that $A$ is promptly simple and that the function $h$ shows this imply that there
exist numbers $x$ and $t>0$ with $x \in (W_{g(e)}[t] \setminus W_{g(e)}[t-1]) \cap A[h(t)]$.
Then there must exist some $s$ with $0<s<t$ such that $x\in U_e[s] \setminus U_e[s-1]$,
and this is possible only if $R_e$ receives attention in stage $s$.
Hence, there must exist a binary string $w$ with $|w|\leq s$ and $0.w < b_{r(s-1)} < 0.w+2^{-|w|}$
and such that the smallest number $y$ with
$2^{-y} < \min\{b_{r(s-1)} - 0.w, 0.w+2^{-|w|} - b_{r(s-1)}\}$
is an element of $W_e[s]$ and satisfies $\varphi_e(y)<s$,
and such that $x$ is the smallest number in $\{0,1,\ldots,s\} \setminus U_{e}[s-1]$
satisfying $A(x)[s]=0$ and $(w,x,0)\in C[s]$.
On the other hand, as $x\in A[h(t)]$, in the algorithm we define $r(s)$ is such a way that for some 
binary string $\widetilde{w}$ we have $0.\widetilde{w} < b_{r(s)} < 0.\widetilde{w}+2^{-|\widetilde{w}|}$ and
$(\widetilde{w},x,1) \in C[r(s)]$.
The two facts $(w,x,0) \in C$ and $(\widetilde{w},x,1) \in C$ imply
that none of the two strings $w$ and $\widetilde{w}$ can be a prefix of the other one.
Hence, the intervals $(0.w,0.w+2^{-|w|})$ and
$(0.\widetilde{w},0.\widetilde{w}+2^{-|\widetilde{w}|})$ have empty intersection.
As the distance of $b_{r(s-1)}$, which is an element of the interval $(0.w,0.w+2^{-|w|})$,
to the endpoints of this interval is larger than $2^{-y}$
and as $b_{r(s)}$ is an element of the interval $(0.\widetilde{w},0.\widetilde{w}+2^{-|\widetilde{w}|})$,
which is disjoint from the interval $(0.w,0.w+2^{-|w|})$,
we conclude that $|b_{r(s-1)} - b_{r(s)}|>2^{-y}$.
But in combination with $\varphi_e(y)<s$ this means that $R_e$ is satisfied after this stage $s$.
Contradiction.
\end{proof}

Corollary~\ref{corollary:nc-not-closed-under-comp-functions}
leaves open the question 
which real numbers can appear as values $f(\alpha)$ for nearly computable numbers $\alpha$
and computable real functions $f$.
Here Theorem~\ref{theorem:promptlysimple} restricts the possibilities.

\begin{corollary}
\label{cor:range-function-nc}
If $\beta$ is a nearly computable number and
$f:\subseteq\IR\to\IR$ a computable function with $\beta\in\dom(f)$
then there does not exist a promptly simple set $A$ with $A\leqT f(\beta)$.
\end{corollary}

Note that there are promptly simple sets~\cite[Exercise XIII.1.10]{Soa87}.

\begin{proof}[{Proof of Corollary~\ref{cor:range-function-nc}}]
If $f:\subseteq\IR\to\IR$ is a computable function and $\beta\in\dom(f)$
then $f(\beta) \leqT \beta$ (compare Ko~\cite{Ko84}).
Hence, if there were a promptly simple set $A$ with $A\leqT f(\beta)$
then the transitivity of $\leqT$ would give us
$A \leqT \beta$, and by Theorem~\ref{theorem:promptlysimple}
this is not possible if $\beta$ is nearly computable.
\end{proof}

\section{Outlook}

There is an interesting proper hierarchy of real closed subfields
between the field $\EC$ of computable numbers
and the field $\NC$ of nearly computable numbers;
see \cite{Jan18} and a forthcoming paper by the authors.

%\bibliography{nearly-computable}
%\bibliographystyle{abbrv}

\end{document}